\documentclass[10pt,leqno,amscd,amssymb,verbatim, url]{amsart}
\usepackage{amsfonts,amssymb,graphicx}
\usepackage{amsmath,amscd}
\newtheorem{thm}{Theorem}[section]
\usepackage{amsfonts,latexsym, curves}
\newtheorem{lem}[thm]{Lemma}
\newtheorem{cor}[thm]{Corollary}
\newtheorem{hyp}[thm]{Hypothesis}
\newtheorem{prop}[thm]{Proposition}

\theoremstyle{definition}

\newtheorem{rem}[thm]{Remark}
\newcommand{\bpict}{\begin{picture}}

\newcommand{\epict}{\end{picture}}

\newtheorem{rems}[thm]{Remarks}

\numberwithin{equation}{thm}

\newcommand{\wH}{{\widetilde H}}

\newcommand{\grExt}{\text{\rm ext}}
\newcommand{\grHom}{\text{\rm hom}}

\newcommand{\head}{\text{\rm head}}

\newcommand{\sR}{{\mathcal R}}

\newcommand{\Fr}{F}
\newcommand{\q}{{\bold q}}
\newcommand{\gr}{\text{\rm gr}}

\newcommand{\Ext}{{\text{\rm Ext}}}

\newcommand{\fa}{{\mathfrak a}}

\renewcommand{\arraystretch}{1.125}

\newcommand{\id}{\operatorname{Id}}

\newcommand{\Amod}{A\mbox{--mod}}
\newcommand{\Bmod}{B\mbox{--mod}}
\newcommand{\fbb}{{\mathfrak b}}
\newcommand{\Agrmod}{A\mbox{--grmod}}
\newcommand{\Hom}{\text{\rm Hom}}

\newcommand{\End}{\operatorname{End}}

\newcommand{\ch}{\operatorname{ch}}

\newcommand{\soc}{\operatorname{soc}}

\newcommand{\rad}{\operatorname{rad}}

\newcommand{\Jan}{\Gamma_{\text{\rm Jan}}}
\newcommand{\Res}{\Gamma_{\text{\rm res}}}
\newcommand{\Resreg}{\Gamma_{\text{\rm res,reg}}}
\newcommand{\Reg}{\Gamma_{\text{\rm reg}}}

\newcommand{\Gmod}{G{\text{\rm --mod}}}

\newcommand{\blist}{\begin{list}{\rom{(\roman{enumi})}}{\setlength
{\leftmarg in}{0em} \setlength{\itemindent}{7ex}
\setlength{\labelsep}{2ex}\setlength{\listparindent}{\parindent}
\usecounter{enumi}}}
\newcommand{\elist}{\end{list}}

\newcommand{\Image}{\text{\rm Im}}
\newcommand{\Forget}{\text{\rm Forget}}
\newcommand{\famod}{{\mathfrak a}{\text{\rm --mod}}}
\newcommand{\fagrmod}{{\mathfrak a}{\text{\rm --grmod}}}

\newcommand{\Extone}{{\rm Ext}^n_{A}(\Delta^r,L)}
\newcommand{\Exttwo}{{\rm Ext}^{n+1}_{A}(\Delta^{0,r},L)}
\newcommand{\Extthree}{{\rm Ext}^n_{\mathfrak{a}}\!(\Delta'^{r}\!,L)}
\newcommand{\Extfour}{{\rm Ext}^{n+1}_{\fa}(\Delta'^{0,r},L)}

\newcommand{\Extfive}{\displaystyle{\bigoplus_{s \geq r+n}}\!\! {\grExt}^n_{\mathfrak{a}}\!(\Delta'^{r}\!,L(s))}

\newcommand{\Extsix}{\displaystyle{\bigoplus_{s \leq (r-1)+n+1}}\!\!\!{\grExt}_{\mathfrak{a}}^{n+1}\!(\Delta'^{0,r}\!,L(s))}

\newcommand{\Extseven}{{\rm Ext}^n_{A}(\Delta^{r+1},L)}

\newcommand{\Exteight}{
{\rm Ext}^{n+1}_{A}(\Delta^{0,r+1},L)}

\newcommand{\Extnine}{
{\rm Ext}_{\mathfrak{a}}^{n}(\Delta'^{r+1},L)}

\newcommand{\Extten}{
{\rm Ext}^{n+1}_{\mathfrak{a}}(\Delta'^{0,r+1},L)}

\newcommand{\Exteleven}{
\displaystyle{\bigoplus_{s \geq r+1+n}}\!\!\!{{\grExt}}_{\mathfrak{a}}^{n}\!(\Delta'^{r+1}\!,L(s))}

\newcommand{\Exttwelve}{
\displaystyle{\bigoplus_{s \leq r+n+1}}\!\!\!{{ \grExt}}_{\mathfrak{a}}^{n+1}\!(\Delta'^{0,r+1}\!,L(s))}
\newcommand{\grAmod}{{\text{\rm gr}}A{{\text{\rm --mod}}}}

\dedicatory{We dedicate this paper to the memory of Ed Cline.}

\begin{document}

 \title[A new approach to the Koszul property in representation theory]{A new approach to the Koszul property in representation theory using graded subalgebras}
\author{Brian J. Parshall}\address{Department of Mathematics \\
University of Virginia\\
Charlottesville, VA 22903} \email{bjp8w@virginia.edu {\text{\rm
(Parshall)}}}
\author{Leonard L. Scott}
\address{Department of Mathematics \\
University of Virginia\\
Charlottesville, VA 22903} \email{lls2l@virginia.edu {\text{\rm
(Scott)}}}

\thanks{Research supported in part by the National Science
Foundation}


\medskip
\medskip
\medskip \maketitle

\begin{abstract} Given a quasi-hereditary algebra $B$, we present conditions which guarantee that the
algebra $\gr B$ obtained by grading $B$ by its radical filtration is Koszul and at the same time inherits the quasi-hereditary
property and other good Lie-theoretic properties that $B$ might possess. The
method involves working with a pair $(A,\fa)$ consisting of a quasi-hereditary algebra $A$ and a (positively) graded subalgebra $\mathfrak a$. The algebra $B$ arises as a quotient $B=A/J$ of $A$ by a defining ideal $J$ 
of $A$. Along the way, we also show that
the standard (Weyl) modules for $B$ have a structure as graded modules for $\fa$. These results are applied to obtain new information about the
finite dimensional algebras (e.~g., the $q$-Schur algebras) which arise as quotients of quantum enveloping algebras. Further applications, perhaps the most penetrating, yield
 results for the
finite dimensional algebras associated to semisimple algebraic groups in positive characteristic $p$. These
results require, at least presently, considerable restrictions on the size of $p$. 
\end{abstract}


\section{Introduction}  In the study of rational representations of a reductive group $G$ over an
algebraically closed field $K$ of characteristic $p>0$, the finite
dimensional
modules with composition factors having highest weights in some fixed finite saturated
set (poset ideal) $\Gamma$ of dominant weights identify (by means of a category equivalence) with the finite
dimensional modules for a finite
dimensional algebra $B$. The homological algebra of $B$ closely parallels
that of $G$; for example, given $M,N\in \Bmod$, $\Ext^\bullet_B(M,N)\cong
\Ext^\bullet_G(M,N)$, where, on the right-hand side, $M,N$ are identified with rational $G$-modules. By
varying $\Gamma$, the representation theory of $G$ can be largely recaptured from that of the algebras $B$.

The algebras $B$ are examples of quasi-hereditary
algebras, which have been extensively studied in their own right, but often with an eye toward
applications to representation theory.  When  $G=GL_n$, the famous Schur algebras arise
 this way \cite{GES}. The representation/cohomology
theory of a (Lusztig) quantum enveloping algebra $U_\zeta$ at a root of unity also can be studied by
means of similarly defined finite dimensional algebras. The $q$-Schur algebras of
Dipper-James \cite{DJ} (see also  \cite{Donkin}, \cite{DDPW}  for exhaustive treatments) are the most
 well-known examples.

This paper presents new results on the structure of the graded algebras 
$$ \gr B :=\bigoplus_{n\geq 0}\rad^nB/\rad^{n+1}B$$
 attached to some of the important algebras $B$ alluded to in the previous
two paragraphs.   One
dominant theme of our own work, often with Ed Cline,
has been a search for Koszul structures in the quasi-hereditary
algebras
 of interest in modular representation theory of algebraic groups. Koszul
algebras
(see  \cite{Priddy}, \cite{BGS},  and also \cite {PS0}) have very strong
homological properties, intertwined
with very strong grading properties. In particular,
a Koszul algebra $B$ is positively graded, and
$B\cong \gr B$. However, it is very difficult, in general, to obtain a good
positive grading
on an algebra $B$, or establish an isomorphism $B\cong \gr B$. Indeed, it is possible
for a finite dimensional algebra to enjoy many standard Lie-theoretic
properties, including all the known homological consequences of Koszulity, yet fail to be Koszul or even to have a positive grading consistent 
with its radical filtration, see \cite[3.2]{CPS5} for a specific example. Fortunately,
in the example given there, the algebra $\gr B$ is Koszul. With this and other insights from \cite{CPS5}
as a starting point, this paper considers the
general question of what good properties
 present in  a finite dimensional algebra $B$ imply that
the graded algebra
 $\gr B$ is Koszul, requiring at the same time that $\gr B$ keep these good
properties (e.~g., quasi-heredity, parity
 conditions,...). Of course, in many circumstances, results about $\gr B$
can then be translated back to $B$.
    Although our methods are entirely algebraic (in the classical sense), the reader may be reminded
of Mumford's notion of ``projective methods" in \cite[p. 137]{Mumford}, where he viewed
the process of forming graded algebras from filtrations (and extracting
useful information from the process) one of the cornerstones of his view of
algebraic geometry. The idea has, of course, occurred to classical
algebraists, too, but it has  been very difficult to say much about $
\gr B$. This is still true of the results in this paper, which often require
 difficult proofs, though we do provide a general framework for
them.

One new ingredient in this general framework   is the use
(when they exist) of  Koszul subalgebras. This 
approach is developed in detail in Part I, where it forms the major theme.  The reader is referred to the introduction of Part I for
a more detailed discussion. Although $B$ itself does often have a Koszul subalgebra, it turns out to be more
flexible to require that such a subalgebra lies inside an algebra $A$ which has $B$ has a homomorphic
image. The basic ingredients in this setup include
 the following: 

\begin{itemize}
\item[(1)] $B$ is a quasi-hereditary algebra and $\Bmod$ has an abstract Kazhdan-Lusztig theory in
the sense of \cite{CPS-1}.
\item[(2)] $B=A/J$ is the quotient of a quasi-hereditary algebra $A$ by a certain defining ideal $J$.  Thus, $B$ has weight poset $\Gamma$, a poset ideal  in the weight poset $\Lambda$ of $A$;
\item[(3)] $A$ has a Koszul subalgebra $\fa$ such that $(\rad \fa)A=\rad A$.
\end{itemize}
The goal, carried out in \S\S6,7, under additional hypotheses, is to show that  $\gr B$ is Koszul, and that a large list of properties of $B$
transfer to $\gr B$. These include the properties of being quasi-hereditary, and of having a module category with an abstract Kazhdan-Lusztig theory.   Furthermore,
$B$ and $\gr B$ share the same Kazhdan-Lusztig polynomials and have the same homological
dual $B^!:=\Ext^\bullet_{B}(B/\rad B, B/\rad B)$. The multiplicities $[\rad^n \Delta(\lambda)/\rad^{n+1}\Delta(\lambda):L(\nu)]$ of irreducible modules $L(\nu)$ in the radical layers of standard (Weyl) modules $\Delta(\lambda)$ for $B$ agree
with the corresponding multiplicities for $\gr B$, and the latter can be calculated as coefficients of
inverse Kazhdan-Lusztig polynomials. In fact, the graded module category of $\gr B$ has the very strong
property of having a graded Kazhdan-Lusztig theory in the sense of \cite{CPS-1}, and the calculation
of graded multiplicities is one of its many consequences.

Another result is the elementary Proposition \ref{KosFact},
which takes $B=A$ (with no quasi-heredity or Kazhdan-Lusztig theory assumption) and proves Koszulity directly for $\gr B$ under simple hypotheses on the subalgebra $\fa$. While these hypotheses do 
not hold for our applications in Part II, the result provides an interesting model, and should at least be of interest
to mathematicians working with finite dimensional algebras.

In our applications, starting with the finite dimensional algebra $B$ associated to an quantum enveloping
algebra or a simple simply connected group, quite a bit of effort is required to properly choose the
pair $(A,\fa)$ with all the required properties (which include (1) -- (3) above as well as  further
homological conditions). This project is carried out in Part II. For example, the Koszulity of the subalgebra $\fa$
comes about from work on Andersen-Jantzen-Soergel \cite{AJS} (and more recently Riche \cite{Riche}).
See the introduction to Part II for more discussion.

The proofs in this paper have independent interest, apart from the results themselves, because of the algebraic
methods involved.\footnote{Using geometric methods, quoting results asserted in \cite{ABG}, it seems likely that the finite dimensional algebras $B$ studied in the quantum case (see \S8 in Part II),
as well as the algebras $\gr B$, can be shown to be Koszul in the regular weight case. However, it seems first necessary to complete some of the arguments in \cite{ABG}. We will return to this in \cite[\S7]{PS6}. These geometric methods are insufficient, however, to obtain the positive characteristic results of \S10. Even in the quantum
case, simply knowing Koszulity of $B$ does not imply any connection with the Koszul algebra
structure  of
$\fa$. In our approach, this is built in, and provides an important theme for this paper and subsequent
work \cite{PS5}. {\it Added Oct. 7, 2011:}  Recently, Shan-Vasserot-Varagnolo (arXiv: 1107.0146) have proved a parabolic-singular duality theorem in the affine Lie algebra case which shows, at least, that the $q$-Schur
algebras are Koszul.  It also likely implies (using \cite[Appendix, \S9]{PS6}) that the algebras $B$ above (of \S8, Part II) are Koszul, even
in the singular case. The proofs are a mixture of algebraic and geometric methods. 
}
In addition, this paper lays the groundwork for several sequels. For example, it plays an
essential role in \cite{PS6}, giving new results on filtrations of $q$-Specht modules for
Hecke algebras.  Similarly, some of the theorems and methods of this paper are used in obtaining new
graded quasi-hereditary algebras over DVRs in \cite{PS5.5}, which in turn are essential in applications to
$p$-filtrations for Weyl modules of algebraic groups planned in \cite{PS5}.

Recently, there has been considerable activity regarding graded algebras. See \cite{BS}, \cite{Turner1},
and \cite{BGS}, together with their references.  All these papers deal with either Koszul or
at least positively graded algebras. For the $\mathbb Z$-graded case, see, for example, \cite{BK}
and its references. 
Our studies here are independent of these works, as is, to some extent, the main thread of our interests.  The authors have had a long involvement in the study of positively graded quasi-hereditary
algebras and corresponding homological questions, see, e.~g., \cite{CPS1a}, \cite{CPS1}, \cite{CPS2}, 
\cite{CPS5}, \cite{PS0}, \cite{PS1}. These papers were part of a general program to understand in an algebraic way some
of the geometric ingredients which led to a proof of the Kazhdan-Lusztig conjecture \cite{BK}, \cite{BB},
and the positive characteristic Lusztig conjecture for large primes \cite{AJS}, \cite{KL}, and \cite{KT}.  Some
of our work was inspired by a preprint of Beilinson-Ginzburg on Koszul algebras, an early version of
\cite{BGS}, which had algebraic consequences beyond the Kazhdan-Lusztig conjecture.  Our efforts hit an obstacle, however, in paper \cite{CPS5}, when we discovered the example \cite[3.2]{CPS5}
mentioned above.  The approach of this paper in some sense pushes beyond that obstacle, and raises 
the prospect of further progress in this program.

We heartily thank the referee for suggestions on improvement of the manuscript. In particular, the
referee suggested changing the original title (``New graded methods in representation theory") to
a more descriptive one. We have tried to do this. As discussed above, either title refers to the process of transferring good properties of $B$ to the graded algebra $\gr B$ (e.~g., quasi-heredity), and then working with this positively graded
algebra to go further.  The existence of the subalgebra $\fa$ makes it possible for this line of attack 
to succeed.  

\vskip.2in
\begin{center}{{\large\bf Part I: Theory}}\end{center}
Part I, which consists of \S\S 2---7, develops some new theoretical techniques for studying the representation
theory of certain finite dimensional algebras (all taken over some fixed based field $K$). As a guide to the reader, we now describe
the approach in Part I in more detail.  In this discussion, and in Part I itself,  $A$ often
denotes a general finite dimensional algebra, which could be either the algebra $A$ or $B$ above, or
even, in some cases, $\fa$.

 As noted in the
introduction, given a finite
dimensional algebra $A$, the positively graded algebra $\gr\,A$ is formed simply by using the
radical series of $A$; see (\ref{natiuralisofortightlygraded}) below. Given a finite dimensional $A$-module $M$, there is an analogous finite dimensional
graded $\gr\,A$-module $\gr\, M$; see (\ref{grMODULE}) below. \S2 introduces this and related elementary
concepts.  If $A$ 
is a (positively) graded algebra, and if $X,Y$ are graded $A$-modules, then the extension groups $\Ext^n_A(X,Y)$, $n=0,1,2,\cdots$, are
computed in the category $A$--mod of {\it ungraded} $A$-modules, while the extension groups $\grExt^n_A(X,Y)$, $n=0,1,2,\cdots$,
are computed in the category $A$--grmod of {\it graded} $A$-modules. (This notation is taken from the
standard reference \cite{BGS}.)   Connecting
the graded ext-groups to the ungraded Ext-groups is easy; just use the simple formula
(\ref{Extgrouprelationstograded}) below.  

Unfortunately, the functor, $\gr:\Amod\to\grAmod$, $M\mapsto\gr M$, from $A$-modules to $\gr A$-modules is not exact. Nevertheless, it is sometimes possible to extend the functor $\gr$ so that it is defined on
extensions, i.~e., elements of $\Ext^1_A(M,N)$. This process, which we call the ``$\gr$-construction"
is studied in \S3.  As far as we know, this idea and the results of this section are completely new.  

Some explanation of the notation in these first sections is required, since the results are
often cast in terms of a finite dimensional algebra denoted $A$ and sometimes in terms of a finite
dimensional algebra denoted $\fa$. The reason for this apparent dichotomy is that, for later applications,
$\fa$ will often be a (very special) subalgebra of $A$. We will need the results of \S3 for both $\fa$
and $A$, and we have tried to state results for the algebra (i.~e., for $\fa$ or $A$) for which the
later application is relevant.

Usually, $\fa$ will be a normal subalgebra of $A$. This means simply that $\fa$ is an augmentated
algebra with augmentation ideal $\fa_+$ such that $A\fa_+=\fa_+A$ (necessarily
a two-sided ideal of $A$). In this case, the quotient algebra $A/\fa_+A$ is usually denoted
$A//\fa$. 

In later results, $\fa$ will also be positively graded, tightly graded (i.~e., $\fa\cong\gr\fa$),
or (especially) Koszul.  The basic
conceptual underpinning of this work aims to study the representation theory of a finite
dimensional algebra $A$ (often with a focus on a quotient algebra $B$) which has a normal Koszul subalgebra $\fa$ which is large in the sense
that $(\rad\,\fa) A=\rad A$. In particular, this means that $A//\fa$ is semisimple. Various other conditions
will be imposed on the pair $(A,\fa)$, but as we see in Part II, these conditions can be verified
in applications.

From \S4 on, the algebra $A$ is often taken to be quasi-hereditary in the sense of \cite{CPS-1}. A more
recent account of the basic material on quasi-hereditary algebras is contained in \cite[Appendix C]{DDPW}.  
The section lays
out the essential parity conditions---originally developed in \cite{CPS1}, \cite{CPS2}, and \cite{CPS5}---that later play an important role.  The most basic condition requires that, if $A$ is a quasi-hereditary
algebra with weight poset $\Lambda$ having a ``length" function $l:\Lambda\to\mathbb Z$, then, stating
just one of the two dual conditions (see \S4),  
  $$\leqno{{\text{\rm (KL)}}} \,\,\,\,\,\Ext^n_A(\Delta(\lambda),L(\nu))\not=0\implies l(\lambda)-l(\nu)\equiv n\,\,\text{\rm mod}\,2,\quad \forall \lambda,\nu\in \Lambda, n\in\mathbb N.$$
 This (KL) condition is a special case of a much stronger condition 
$$\leqno{\text{\rm (SKL$'$)}}\,\,\,\,\, \Ext^n_A(\rad^i\Delta(\lambda), L(\nu))\not=0\implies l(\lambda)-l(\nu)-i\equiv n\,\,{\text{\rm mod}}\,\,2,\quad\forall \lambda,\nu\in\Lambda, n,i\in\mathbb N. $$
In these expressions, $\Delta(\lambda)$ (resp., $L(\nu)$) is the standard (resp., irreducible) module
associated to the weight $\lambda$ (resp., $\nu$) in $\Lambda$. As indicated in \S4, there are dual versions of
these two conditions, involving the irreducible modules $L(\nu)$  and costandard modules $\nabla(\lambda)$, which are automatic when the above versions hold and the category $\Amod$ has a duality.\footnote{A duality
on an abelian category $\mathcal C$ is a contravariant equivalence ${\mathfrak d}:{\mathcal C}
\to{\mathcal C}$. When $L\in\mathcal C$ is irreducible, we usually require ${\mathfrak d}L\cong L$.
In categories which allow for ${\mathbb Z}$-gradings of objects, we require instead that 
${\mathfrak d}L\cong L$ when $L$ is irreducible of pure grade 0. Also assume, if $M$ is any graded object, then
${\mathfrak d}M(n)=({\mathfrak d}M)(-n)$,
  where $M \mapsto M(n)$
is the ``shift" operator for the graded structure (given in our context by
$M(n)_i =M_{n-i}$, for integers $n, i$).}
  Of course, for an arbitrary quasi-hereditary algebra, the properties (KL) or (SKL$^\prime$) may not hold.  In fact, the
validity of (KL), or of (SKL$'$), has very strong consequences. For example, in standard examples involving
semisimple algebraic groups in positive characteristic, condition
(KL) is equivalent to the validity of the Lusztig character formula; see \cite{CPS1}. Also, a main result, proved in 
\cite{CPS5}, shows that when  (SKL$'$) holds, then the algebra $\gr A$ is Koszul, quasi-hereditary, and retains the parity properties enjoyed by $A$, as well as its Kazhdan-Lusztig polynomials. The algebra $\gr A$ has
a graded Kazhdan-Lusztig theory in the sense of \cite{CPS-1}.

\S6 concerns a pair $(A,\fa)$ in which $A$ is quasi-hereditary with weight poset $\Lambda$ and $\fa$
is a tightly graded subalgebra. The conditions on this pair are laid out at the start of \S6. (The material
in \S5 is a technical digression, capitalizing on the fact that, in our applications, the algebras are homomorphic
images of certain Hopf algebras.)   
Theorem \ref{MainInfinitesimalThm} establishes that certain of the standard
modules $\Delta(\lambda)$ for $A$, as well as other $A$-modules, have a natural structure as graded $\fa$-modules, generated by
their terms in grade 0. The main tool in establishing this important result is the difficult
 Lemma \ref{bigdiagramlemma}, which inductively approaches $\fa$-gradings one radical layer at a time. Theorem \ref{grBisQHA}  establishes that, if $\Gamma$ is a poset
 ideal in $\Lambda$ and if $B=A/J$ is the associated quotient quasi-hereditary algebra,  then
 $\gr B$ is quasi-hereditary (subject to a condition on $\Gamma$).

Finally, \S7 is the heart of Part I.  In the main result, Theorem \ref{Koszultheorem}, we have a pair $(A,\fa)$
as above, but with $\fa$ Koszul. There, $B$ is a quasi-hereditary quotient algebra of $A$ associated to
a poset ideal $\Gamma$. It is assumed that $\Bmod$ satisfies the  (KL) property above and a further
homological condition stated in Hypothesis \ref{section7hypothesis}(2) is needed. The conclusion is that
$B$ satisfies the (SKL$'$) condition, and, so $\gr B$ is Koszul and has other strong properties
as discussed two
paragraphs above (as well as the introduction). See Part II
(including its introduction) for further discussion of how the Koszulity property hypothesized for $\fa$ is known in the cases of interest there.
 
\vskip.2in
\section{Preliminaries on gradings}
We work with finite dimensional algebras $A$ (or $\fa$) over a field $K$.  Let $A$-mod be the category of left $A$-modules which are finite dimensional over $K$. (All ``$A$-modules" in this paper will be assumed to 
be finite dimensional, unless otherwise indicated.) There is given a set $\Lambda$ in bijection with the set Irr($A$) of isomorphism
 classes of irreducible $A$-modules.   For $\lambda \in \Lambda$, let $L(\lambda)$ be an irreducible module whose isomorphism class corresponds to $\lambda$.

 Often  $A$ is positively graded, i.~e.,
 $$A=\bigoplus_{n\in\mathbb N} A_n,\quad A_mA_n\subseteq A_{m+n}, \quad (m,n\in{\mathbb N}).$$
 (The term ``graded algebra" will always mean ``positively graded algebra.")
 Let $A$--grmod be the category of left $A$-modules $M$ which are finite dimensional $K$-modules and which are $\mathbb Z$-graded, i.~e.,
$ M=\bigoplus_{n\in
 \mathbb Z}M_n$ over $A$.
 Thus, $M_n=0$ if $|n| >\!\!> 0$, and, for $m,n\in\mathbb Z$, $A_mM_n\subseteq M_{m+n}$. If  $M,N\in A$--grmod,
 $\Hom_{A{\text{\rm{--grmod}}}}(M,N)\equiv \grHom _A(M,N)$ is the $R$-module of $A$-module homomorphisms $f:M\to N$ such that $f(M_n)\subseteq N_n$, for all integers $n$.

 Given a graded $A$-module $M$, $M_n$ is called the $n$th grade of $M$; it is
 naturally an $A_0$-module, and, therefore, can be regarded as an $A$-module through the surjective map $A\twoheadrightarrow A/(\sum_{i>0}A_i)\cong A_0$. For $M\in A$--grmod and $r\in \mathbb Z$, let $M(r)\in A$--grmod be the graded $A$-module obtained by
 shifting the grading on $M$ by $r$ places to the right; i.~e., $M(r)_n=M_{n-r}$.  Then
 \begin{equation}\label{iso1}\Hom_A(M,N)=\bigoplus_{r\in\mathbb Z}\grHom _A(M,N(r)).\end{equation}
An irreducible $A$-module $L$ is regarded as an irreducible graded $A$-module concentrated in grade $0$.

We say $M\in\Agrmod$ is generated in grade $i$ if $AM_i=M$.  Let
 $\Agrmod_i$ be the full subcategory of $\Agrmod$ consisting of modules generated in grade $i$.

When $A$ is graded, the category $A$--grmod is an abelian category with enough projective objects.
The bifunctors $\grExt^n_A(-,-):\Agrmod\times\Agrmod\to K{\text{\rm --mod}}$, $n=0,1,\cdots$,  are therefore defined. As with (\ref{iso1}), these
functors are related to the bifunctors $\Ext^\bullet_A(-,-)$ in $\Amod$ by an isomorphism
\begin{equation}\label{Extgrouprelationstograded}
\Ext^\bullet_A(M,N)\cong\bigoplus_{r\in\mathbb Z}\grExt^\bullet_A(M,N(r)),\quad M,N\in\Agrmod\end{equation}
preserving homological degrees.

Given a graded algebra $A$, let $\Forget=\Forget_A:\Agrmod\to\Amod$ be the exact ``forgetful functor" which assigns to any $M\in\Agrmod$ the underlying $A$-module
in which the grading is ignored.\footnote{While useful in this preliminary section, we will generally be more informal with the forgetful
 functor, using the same symbol $M$ for both a graded or ungraded module, as determined by context.} For $P\in\Agrmod$, $P$ is projective if and only if $\text{Forget}\, P$ is projective in $\Amod$.

\smallskip
 Let $A$ be a finite dimensional $K$-algebra. We do not necessarily assume that $A$ is graded. Let $\rad\,A$ be the radical of $A$ and, for  $r\in\mathbb N$, write $\rad^rA=(\rad A)^r$. Similarly, if $M\in A$-mod, put
$\rad^rM=\rad_A^rM:=(\rad^rA)M$. More generally, given integers
$r,s$ satisfying $0\leq r\leq s$, put
\begin{equation}\label{Mrsnotation} M^{r,s}=\rad^rM/\rad^sM\in
\Amod.\end{equation}
Thus, $M^{0,s}=M/\rad^sM$. We sometimes write $M^r=M^{r,\infty}$ for $\rad^rM$.
Also, let
$0=\soc_0M\subseteq\soc_{-1} M\subseteq \cdots$
be the socle series of $M$. Thus, $\soc_{-1}M$ is the socle of $M$, while for $i>1$, $\soc_{-i}M\subseteq M$ is the unique submodule containing $\soc_{-i+1}M$ such that $\soc_{-i}M/\soc_{-i+1}M=\soc(M/\soc_{-i+1}M)$.

A graded algebra $A=\bigoplus_{n\in\mathbb N}A_n$ is
tightly graded if $A_0$ is semisimple, and, for $n\geq 1$,
$A_n=A_1^n$. It is Koszul provided that, if $L,L'$ are irreducible modules with grading concentrated in grade 0, then
$$\grExt^n_A(L,L'(r))\not=0\implies n=r.$$ Equivalently, if $P_\bullet\twoheadrightarrow L$ is a minimal projective resolution in $\Agrmod$, then the
head of $P_i$, $i\geq 0$, is generated in grade $i$. Any Koszul algebra is tightly graded.

 For any positively
graded algebra $A$, $M\in A$-grmod $_0$ is said be linear, or to have a linear resolution, if there graded projective resolution $P_\bullet\twoheadrightarrow M$ in which each $P_i$ generated in grade $i$.  (Such
a resolution is necessarily minimal.)  See \cite{Mazor} for more on this topic. Thus, $A$ is
Koszul provided every irreducible module $L$ (viewed as a graded module concentrated in grade 0) is linear.
Even if $A$ is Koszul, not every $M\in A$-grmod$_0$ need be linear. However, linearity for $M$ is equivalent
to an ext-vanishing condition as stated for $L$.

Returning to the general finite dimensional algebra $A$, the associated graded algebra
\begin{equation}\label{gradedalgebra} \gr A  :=\bigoplus_{i\geq 0}\rad^iA/\rad^{i+1}A=\bigoplus_{i\geq 0} A^{i,i+1}\end{equation}
is tightly graded. If $A=\bigoplus_{n\geq 0} A_n$ is itself tightly graded, there is a natural isomorphism
\begin{equation}\label{natiuralisofortightlygraded} \iota_A:A\overset\sim\longrightarrow\gr A\end{equation}
of graded algebras which send $x\in A_i\subseteq\rad^iA$ to its image $[x]$ in $\rad^iA/\rad^{i+1}A$. It induces  a category
 equivalence $\iota_{A *}:\gr A{\text{--\rm grmod}}\overset\sim\longrightarrow \Agrmod$.
If $M\in\Amod$, putting
\begin{equation}\label{grMODULE}\gr M=\gr_AM:=\bigoplus_{n\geq 0}\rad^nM/\rad^{n+1}M\in\gr\Agrmod_0\end{equation}
defines an additive functor $\gr:\Amod\to\gr A{\text{\rm{--grmod}}}_0$, i.~e., a functor from the category of $A$-modules (generated in grade 0) to the category of graded modules for
the tightly graded algebra $\gr A$. While not left or right exact, $\gr$ preserves surjections. Additionally, if $P\twoheadrightarrow M$ is a projective cover in $\Amod$, then $\gr P\twoheadrightarrow\gr M$ is a projective cover in $\gr\Agrmod$. Since $(\gr A)/(\rad\gr A)\cong A/\rad A$, the module categories $\gr\Amod$ and $\Amod$ for the algebras
$\gr A$ and $A$ have the same irreducible modules (and completely reducible)  modules.

Also, define
$\gr M=\gr^\bullet_AM:=\bigoplus_{i\leq 0} \soc_{i-1}M/\soc_iM\in\gr\Agrmod.$
Clearly, $\gr^\bullet:\Amod\to\gr\Agrmod$ is  an additive functor, preserving injections. Both $\gr $ and $\gr^\bullet$ make sense for right $A$-modules $M$ (resulting in graded
right $\gr A$-modules). Since the linear dual $(-)^*$ on modules reverses left and right modules,  $(\gr M)^* = \gr^\bullet M^*$ and $(\gr^\bullet M)^*=\gr M^*$, for any left or right $A$-module $M$.

A short exact sequence $0\to L\to M\to N\to 0$ in $\Amod$ defines a short exact
 sequence $0\to\gr^\#L\to \gr M\to \gr N\to 0$ in $\gr\Agrmod$, in which $\gr^\#L=\gr^\#_{M}L$ (after regarding $L$ as
 a submodule of $A$) is given explicitly as
\begin{equation}\label{grpounddefined} \gr^\#L= \bigoplus_{s\geq 0}(L\cap\rad^sM)/(L\cap\rad^{s+1}M)\in\gr\Agrmod.\end{equation}
 If $L$ is completely reducible, then $\Forget(\gr^\#L)\cong L$, although the isomorphism is not canonical.

\begin{prop}\label{basictightprop}
 Suppose that $\fa$ is a tightly graded algebra.
 There is a natural isomorphism
 $$\iota_{\fa *}\circ\gr\circ\Forget|_{\fagrmod_0}\cong\id_{\fa{\text{\rm{--grmod}}}_0}$$
 of functors. In other words (and more informally), if $M\in\fagrmod_0$, and if $\gr M$ is given a graded $\fa$-structure via the isomorphism
$\iota_{\fa}:\fa\overset\sim\to\gr\, \fa$, then $M\cong\gr M$.\end{prop}

\begin{proof} The isomorphism $\iota_\fa:\fa\to\gr\,\fa$ sends $x\in \fa_i$ to $[x]_i\in\rad^i\fa/\rad^{i+1}\fa$. Define a map $M\to\gr M$ similarly.
For $m\in M_j$, we have $xm\in M_{i+j}$, and
$xm\mapsto [xm]_{i+j}=[x]_i[m]_j=x[m_j],\forall i,j\in\mathbb Z.$
So our map is a homomorphism. It is clearly bijective, and the lemma is proved.\end{proof}

\begin{cor} Let $\fa$ be a tightly graded algebra. If $M,M'\in\fagrmod_0$ satisfy $\Forget(M)\cong\Forget(M')$ (i.~e., if $M$ and $M'$ are isomorphic as
ungraded $\fa$-modules),  then $M\cong M'$ in $\fagrmod$.  \end{cor}

For another argument for this result, without the hypothesis of tight grading, see \cite[Thm. 4.1]{GG}.

\section{The $\gr$-functor and extensions}

We work with two finite dimensional algebras $\fa$ and $A$ over a field $K$, with $\fa$ usually a subalgebra of $A$.   The functor $\gr=\gr_\fa:\famod\to\gr\,\fagrmod$
defined in \S2, can, in favorable circumstances,
be defined on extensions, i.~e., on elements in $\Ext^1_A(M,L)$ for certain $A$-modules $M,L$.

Let $M\in\Amod$ and let $0\to\Omega\to P\overset\nu\to M\to 0$  be a short exact sequence in which $\nu$ is a
projective cover of $M$. We obtain a commutative diagram with exact rows in $\gr A$-grmod:

\begin{equation}\label{commdiagram}
\begin{array}{ccccccccc}

0 & \longrightarrow & \mbox{gr}^\#\Omega & \longrightarrow & \mbox{gr }P & \stackrel{{\rm gr \,}\nu}{\longrightarrow} & \mbox{gr } M & \longrightarrow & 0

\\[1mm]

&& \begin{picture}(5,20)

\put(0,20){\vector(0,-1){20}}

\put(0,20){\vector(0,-1){16}}

\end{picture}

&&

\begin{picture}(5,20)

\put(0,20){\vector(0,-1){20}}

\put(0,20){\vector(0,-1){16}}

\end{picture}

&&

\begin{picture}(2,20)

\put(0,20){\line(0,-1){20}}

\put(2,20){\line(0,-1){20}}

\end{picture}

&& \\[1mm]

0 & \longrightarrow & \mbox{gr}^\#(\Omega/\mbox{rad }\Omega) & \longrightarrow & \mbox{gr}(P/\mbox{rad }\Omega) & \longrightarrow & \mbox{gr } M & \longrightarrow & 0

\\[1mm]

&&

\begin{picture}(16,20)

\put(0,10){$\wr$}

\put(5,20){\vector(0,-1){20}}

\put(8,10){$\omega$}

\end{picture}

&&&&&&

\\[1mm]

&& \Omega/\mbox{rad } \Omega

&&&&&&

\end{array}
\end{equation}

\noindent
in which ${\gr\nu}$ is a projective cover map.  Here $(\gr^\#\Omega)_i=(\Omega\cap\rad^iP)/(\Omega\cap\rad^{i+1}P)$, $i\geq 0$, and $\gr^\#(\Omega/\rad\Omega)$ is defined similarly. The second row is induced by the exact sequence
$0\to\Omega/\rad\Omega\to P/\rad\Omega\to M\to 0$.
The surjection $\gr P\twoheadrightarrow \gr(P/\rad\Omega)$ induces the top left vertical surjection.
    The isomorphism $\omega$ is fixed, but is not canonical. We forget the gradings, and regard (\ref{commdiagram}) as a
    diagram in $\gr\Amod$. Thus, for any completely reducible $A$-module $L$,  there is an injection
\begin{equation}\label{injection}\Ext^1_A(M,L)\cong\Hom_A(\Omega,L)\cong\Hom_A(\Omega/\rad\Omega,L)\hookrightarrow\Hom_{\gr A}(\gr^\#\Omega,L)\cong\Ext^1_{\gr A}(\gr M,L),\end{equation}
natural in $L$. Thus, we obtain a  natural transformation
\begin{equation}\label{naturaltransformation} \gr=\gr_{A,M,\nu,\omega}:\Ext^1_A(M,-)|_{\text{\rm completely reducible}}\longrightarrow\Ext^1_{\gr A}(\gr M,-)|_{\text{\rm completely reducible}}\end{equation}
of functors on the category of completely reducible $A$-modules; it depends
on the projective cover $\nu:P\twoheadrightarrow M$ and the isomorphism $\omega$ in (\ref{commdiagram}).

\begin{prop}\label{grProp1} Suppose that $\fa$ is a subalgebra of a finite dimensional algebra $A$ with $(\rad\fa)A=\rad A$, and $M$ is an $A$-module with a projective cover
$P\overset\nu\twoheadrightarrow M$ in which $P$ is also projective as an $\fa$-module. Then the choices $\nu,\omega$ above for $A$ and $M$ may also be used
for $\fa$ and the restriction of $M$ to $\fa$. For each completely reducible $A$-module $L$, we have a commutative diagram
 $$\begin{CD}  \Ext^1_A(M,L) @>{\gr}>> \Ext^1_{\gr A}(\gr M,L) \\
@VVV @VVV\\
  \Ext^1_\fa(M,L) @>{\gr}>>  \Ext^1_{\gr\fa}(\gr N,L) \end{CD}$$
where the horizontal maps arise from the $\gr$-construction and the vertical maps are induced by ``restriction"
through the algebra homomorphisms $\fa\to A$ and $\gr\,\fa\to\gr\,A$.\end{prop}

\begin{proof} The condition $(\rad\fa) A=\rad A$ implies by induction on $n$ that $\rad^nA=\rad^n\fa A:=(\rad\,\fa)^nA$. (Observe that $\rad^2A=(\rad\fa) A (\rad A)=(\rad\fa)(\rad A)
 =(\rad\fa)(\rad\fa)A=(\rad\fa)^2A$, etc.) Therefore, given an $A$-module $M$, $\rad^n_AM=\rad^n_\fa N$, so
that $\gr_AM=\gr_\fa M$. In particular, a completely reducible $A$-module remains completely reducible upon restriction to $\fa$. Also, if $M\in\Amod$,
then the head of $M$ as an $A$-module becomes upon restriction to $\fa$ the head of $M|_\fa$.\footnote{
Applied to extensions of $M$ by $L$ (or of $\gr M$ by $L$), this property of heads implies both
vertical restriction maps in the display of the proposition are injective. See Proposition \ref{elemresult}.} 

 Thus, since $\nu:P\twoheadrightarrow M$ is a projective cover in $\Amod$ and  $P|_\fa$ is projective, $\nu$
is also a projective cover in $\famod$. It follows that if the isomorphism $\omega$ in (\ref{commdiagram}) is chosen for $\Amod$, this isomorphism  also works for $\famod$. The commutative diagram thus follows
from the naturality of the restriction maps $\Hom_A(-,-)\to\Hom_\fa(-,-)$ and $\Hom_{\gr\, A}(-,-)\to
\Hom_{\gr\,\fa}(-,-)$.\end{proof}

Let $\fa$ be tightly graded and $M\in\fagrmod_0$.  Taking the projective
cover $P\overset\nu\longrightarrow M$ in $\fagrmod$, both $\Omega$ and $\Omega/\rad\Omega$ belong to $\fagrmod$, and
the isomorphism $\omega$ in (\ref{commdiagram}) can be assumed to be in $\fagrmod$. Also, $\Omega\cong\gr^\#\Omega$ in $\fagrmod$. By Proposition
\ref{basictightprop}, there are isomorphisms $P\cong\gr P$ and $M\cong\gr P$ compatible with the isomorphism $\iota_\fa$. Putting all this together,
we get the following result.
\begin{prop}\label{grProp2} Suppose that the finite dimensional algebra $\fa$ is tightly graded and that $M\in\fagrmod_0$.  For each completely reducible $\fa$-module $L$ with
a graded $\fa$-module structure concentrated in grade 0, and each integer $s$, we have a commutative diagram

$$\begin{CD} \Ext^1_\fa(M,L) @>{\sim}>> \Ext^1_{\gr\fa}(\gr M,L) \\
@AAA @AAA \\
\grExt^1_\fa(M,L(s)) @>{\sim}>> \grExt^1_{\gr\fa}(\gr M,L(s))\end{CD}
$$
in which the vertical maps are the natural inclusions defined by (\ref{Extgrouprelationstograded}), the top horizontal map is an isomorphism given by the $\gr$-construction (determined by $\Omega$ and $\nu$ described above), and the bottom
horizontal map is the natural identification induced by Proposition \ref{basictightprop}. \end{prop}

\begin{rem}\label{grRem1} The proposition provides a (largely obvious) interpretation of the $\gr$-construction in the  special case in which $\fa$ is tightly graded, $M\in\fagrmod_0$, and the extension in $\Ext^1_\fa(M,L)$ arises from
a graded extension  $0\to L(s)\to E\overset\mu\longrightarrow M\to 0$. It says that the result of applying the
$\gr$-construction to this element of $\Ext^1_\fa(M,L)$ is (if nonzero) just the element of $\Ext^1_{\gr\fa}(\gr M,L)$ arising from the
extension
 $0\to L(s)\longrightarrow \gr E\overset{\gr\mu}\longrightarrow \gr M\to 0$
in $\gr\fagrmod$. This is one of the few explicit interpretations we have for the $\gr$-construction. Also, no projective
cover is needed, nor are any other choices. Simply apply $\gr$ to the (ungraded                                               version of) $\mu$, giving a surjection;  the $\gr$-construction is obtained passing to
a natural kernel diagram
using $L$.
(This works  even when the original extension is split, though $E$ will not be generated in
degree $0$ when $s \neq 0$, in the split case.) One cannot hope for such a simple interpretation in general, since this procedure always results in an extension obtained by forgetting the grading on a graded  extension. No  such functor could be additive, if  extensions
  in different degrees by the same irreducible were to exist in its    image, for                                                                                  a fixed $M$   and $\gr M$.    \end{rem}

\begin{thm}\label{grThm1} Suppose $A$ is a finite dimensional algebra over a field and that $P\in\Amod$ is projective. Then, for each non-negative integer $r$ and each completely
reducible $A$-module $L$, the $\gr$-construction induces an isomorphism
$$\Ext^1_A(P/\rad^rP,L)\overset\sim\longrightarrow\Ext^1_{\gr A}(\gr(P/\rad^rP),L).$$
\end{thm}
\begin{proof} The projective cover $P\overset\nu\twoheadrightarrow P/\rad^r P$ has kernel $\Omega=\rad^rP$. Clearly, in this case, the kernel
 $\gr^\#\Omega$ of $\gr\nu$ is isomorphic to $(\gr\Omega)(r)$. Note that $(\Omega/\rad\Omega)(r)$ is the head of
 $(\gr\Omega)(r)$. Consequently,  the surjection $\gr^\#\Omega\twoheadrightarrow\gr^\#(\Omega/\rad\Omega)\overset\omega\longrightarrow\Omega/\rad\Omega$ induces
 an isomorphism on heads, and so an isomorphism $\Hom_A(\Omega,L)\overset\sim\longrightarrow\Hom_{\gr A}(\gr^\#\Omega,L)$, proving the theorem.\end{proof}

\begin{cor}\label{grCor2} Let $A$ be any finite dimensional algebra, and suppose that $D$ and $L$ are completely reducible $A$-modules. Then the $\gr$-construction
gives an isomorphism
$\gr:\Ext^1_{A}(D,L)\overset\sim\to\Ext^1_{\gr A}(D,L).$
\end{cor}

\begin{rems}\label{grRem2} (a) Any choice of the $\gr$-construction may be used above.  We have used a cover $\nu$ with kernel $\Omega=\rad^rP$, but all such covers are isomorphic as maps. Thus, the heads of $\Omega$ and $\gr^\#\Omega$ must be isomorphic for the kernel
 $\Omega$ of any cover $\nu$, and any surjection
$\gr^\#\Omega\to\gr^\#(\Omega/\rad\Omega)\overset\omega\longrightarrow\Omega/\rad\Omega$ must induce an isomorphism on heads, leading to the isomorphism of
the theorem.

(b) We also have an isomorphism
$\Ext^1_{\gr A}(\gr(P/\rad^rP),L)\cong\grExt^1_{\gr A}(\gr (P/\rad^rP),L(r)),$
if $L$ is given a pure graded structure in grade $0$. (This follows because $\gr^\#\Omega$ is generated in grade $r$.)

(c) The proof shows that $\omega:\gr^\#\Omega/\rad\Omega\to\Omega/\rad\Omega$ may be taken to be the identity map on $\rad^rP/\rad^{r
+1}P$. Thus, if $0\to L\to E\overset\mu\longrightarrow P/\rad^rP\to 0$ is an extension, then its image under the $\gr$-construction is
represented by an extension $0\to L\to\gr E\overset{\gr\mu}\longrightarrow\gr(P/\rad^rP)\to 0$. (Indeed, if the first extension is viewed as a push-out of
$f:\Omega\to L$, the second is the push-out of $\gr f$, if the gradings are ignored.) A similar phenomenon occurs in Remark \ref{grRem1}. In both cases, the extensions
over $\gr A$ that are involved are graded, with $L$ given some pure grade. (See (b) above.) Another similar situation to   the present one   occurs in Proposition \ref{grProp3}(c); see the remark following its proof.
\end{rems}

 In the following result,  part (b) shows that, while the algebras $\fa$ and $\gr A$ need not be Morita
equivalent,  $\fa$ can strongly influence the homological algebra of $\gr A$. This is a dominant theme in this paper.

\begin{thm}\label{movinglemma} Let $\fa$ be a tightly graded subalgebra of the finite dimensional algebra $A$ such that $(\rad\fa)A=\rad A$. Let $M\in\Amod$ be such that $M|_\fa\in\fagrmod_0$, and assume that $M$ has a projective cover $P\overset\nu\twoheadrightarrow M$ in $\Amod$ in which $P|_{\fa}$ is projective.  Also, assume that $A_0$ is a Wedderburn complement for $A$ containing $\fa_0$, and that
$M_0$ is $A_0$-stable. Let $L$ be a completely reducible $A$-module.

(a) The projective cover $P\overset\nu\twoheadrightarrow M$ can be assumed, upon restriction to $\fa$, to be a projective cover in $\fagrmod$. Also,
$P|_{\fa}\in\fagrmod_0$.

(b) Any choice of the $\gr$-construction gives an isomorphism
$$\gr:\Ext^1_A(M,L)\overset\sim\to\Ext^1_{\gr A}(\gr M,L).$$\end{thm}

\begin{proof} We first prove (a). Observe that
$$A\otimes_{A_0}M_0\longrightarrow M,\quad a\otimes m
\mapsto am,$$
 is a projective
cover of $M$ in $\Amod$. Similarly, $\fa\otimes_{\fa_0}M_0\to M$, $x\otimes m\mapsto xm$, is a $\fa_0$-projective cover. However, the map $\fa\otimes_{\fa_0}M_0\to A\otimes_{A_0}M_0$ is surjective, since it covers the head of $A\otimes_{A_0}M_0$.
But, since some projective cover $P$ of $M$ as an $A$-module is also $\fa$-projective, all such covers must be $\fa$-projective. Hence, the map
$\fa\otimes_{\fa_0}M_0\to A\otimes_{A_0}M_0$ is a split surjection, and thus an isomorphism of $\fa$-modules (since the $\fa$-heads of $A\otimes_{A_0}M_0$ and
$\fa\otimes_{\fa_0}M_0$ are isomorphic
to $M_0$).

Now consider (b). Since any $\gr:\Ext^1_A(M,L)\to\Ext^1_{\gr A}(\gr M,L)$ is an injection, it will be an isomorphism if
and only if it is an isomorphism for some particular choice of $\nu$ and $\omega$. Thus, by (a), we can assume that $P\overset\nu\to
M$ restricts to a projective cover in $\fagrmod$. In the commutative diagram (\ref{commdiagram}), we can take $\gr$ to be
either $\gr_\fa$ or $\gr_A$, and $\rad$ to be $\rad_\fa$ or $\rad_A$, without changing the spaces. But at the level
of graded $\fa$-modules, it is obvious that the kernel of $\gr^\#\Omega\twoheadrightarrow\gr^\#(\Omega/\rad\Omega)$ is precisely
$\rad(\gr^\#\Omega)$. Thus, the mapping $\gr:\Ext^1_A(M,L)\to\Ext^1_{\gr A}(\gr M,L)$ is an isomorphism.
 \end{proof}

We conclude this section by recording the elementary result mentioned in footnote 3. It is interesting that
the hypotheses (and thus the conclusions) remain valid when $A$ is replaced by an algebra $B$ which
is a homomorphic image of $A$.

\begin{prop}\label{elemresult} Let $\fa\to A$ be a homomorphism of algebras such that, if $A$ is regarded
as a left $\fa$-module, then $(\rad \fa)A=\rad A$. Then, for $A$-modules $M,L$ with $L$ completely
reducible, the natural map $\Ext^1_A(M,L)\to\Ext^1_\fa(M,L)$ is injective. \end{prop}
\begin{proof} It suffices to treat the case in which $L$ is irreducible. If $(^*)\,\,0\to L\to E\to M\to 0$ is a
non-split extension of $M$ by $L$ in $\Amod$, then $L\subseteq\rad E$. But $\rad E=(\rad A)E=
(\rad a)AE=\rad_\fa E$, the radical of $E$ as an $\fa$-module. This means that (*) is a non-split extension
of $M$ by $L$ when the modules are restricted to $\fa$. This proves the required injectivity. \end{proof}

\section{Quasi-hereditary algebras}
 Let $A$ be a quasi-hereditary
algebra (QHA) over a field $K$ with finite (weight) poset $\Lambda$.  For $\lambda\in\Lambda$, let $L(\lambda)$, $\Delta(\lambda)$, $\nabla(\lambda)$ are the corresponding
irreducible, standard, and costandard modules, respectively, in the highest weight category (HWC) $\Amod$. Thus,
$\Delta(\lambda)$ (resp., $\nabla(\lambda)$) has head (resp., socle) isomorphic to $L(\lambda)$, and other composition factors
$L(\mu)$ satisfy $\mu<\lambda$ in the poset $\Lambda$. Assume that $A$ is {\it split}, i.~e.,
$\dim\End_A(L(\lambda))\cong K$ for all $\lambda\in\Lambda$. Let $P(\lambda)\twoheadrightarrow L(\lambda)$
(resp., $\nabla(\lambda)\hookrightarrow Q(\lambda)$) be the projective cover (resp., injective envelope) of $L(\lambda)$.   Let $\Amod(\Delta)$ (resp., $\Amod(\nabla)$) be the full
subcategory of $\Amod$ consisting of modules with a $\Delta$-filtration (resp., $\nabla$-filtration), i.~e., have filtrations
with sections $\Delta(\lambda)$ (resp., $\nabla(\lambda)$), $\lambda\in\Lambda$.   Thus,  $P(\lambda)\in\Amod(\Delta)$ and $Q(\lambda)\in\Amod(\nabla)$.
In any $\Delta$-filtration (resp., $\nabla$-filtration) of $P(\lambda)$ (resp., $Q(\lambda)$), the top (resp., bottom) section
is isomorphic to $\Delta(\lambda)$ (resp., $\nabla(\lambda)$); other sections are isomorphic to $\Delta(\mu)$ (resp., $\nabla(\mu)$)
for some $\mu>\lambda$.

If $\lambda,\mu\in\Lambda$ and $n\in\mathbb N$, then, by \cite[Lemma 2.2]{CPS1},
\begin{equation}\label{elemordering}
 \dim\Ext^n_A(\Delta(\lambda),\nabla(\mu))=\delta_{\lambda,\mu}\delta_{n,0}. \end{equation}

For a nonempty poset ideal  $\Gamma$ in $\Lambda$,  let $\Amod[\Gamma]$ is the full subcategory of the
module category $\Amod$ consisting of
modules all of whose composition factors $L(\gamma)$ satisfy $\gamma\in\Gamma$. There exists an idempotent ideal (sometimes called a defining ideal) $J$ of $A$ such that
$\Amod[\Gamma]$ is Morita equivalent to $A/J$--mod.  In addition, $A/J$ is a QHA. The exact inclusion functor 
$i_*:\Amod[\Gamma]\to\Amod$ (which is just inflation through the map $A\to A/J$) admits a
left adjoint $i^*:\Amod\to\Amod[\Gamma]$ and a right adjoint $i^!:\Amod\to\Amod[\Gamma].$ Explicitly, given $M\in\Amod$, $i^*M$ (resp.,
$i^!M$) is the largest quotient module (resp., submodule) of $M$ lying in $\Amod[\Gamma]$. It will be convenient, given $M\in\Amod$, to denote
$i^*M$ also by $M_\Gamma$. A basic property states that, given $M,N\in\Amod[\Gamma]$, there is a natural isomorphism (preserving grades)
\begin{equation}\label{derivedembedding} \Ext^\bullet_{A/J}(M,N)\cong\Ext^\bullet_A(i_*M,i_*N).\end{equation}
Generally, we denote $i_*M$ simply by $M$, so that $\Ext^\bullet_{A/J}(M,N)\cong\Ext^\bullet_A(M,N)$.

The QHA algebra $A/J$ above is often denoted simply $A_\Gamma$. In fact, it is the largest quotient module
of $A$ (regarded as a left $A$-module) whose composition factors $L(\gamma)$ satisfy $\gamma\in\Gamma$.

\begin{prop}\label{grProp3} Let $P$ be a projective module for a finite dimensional algebra $A$. Let $L\in\Amod$ be completely reducible and
 let $r\geq 0$.

(a) The natural map $\Ext^1_A(P/\rad^rP,L)\to\Ext^1_A(\rad^{r-1}P/\rad^rP,L)$ induced by the
inclusion $\rad^{r-1}P/\rad^rP\hookrightarrow P/\rad^rP$ is injective.

(b) If $A$ is a QHA, $\Gamma$  is a poset ideal in $\Lambda$, and all weights of composition factors of $L$ lie
in $\Gamma$, then the natural map
$\begin{CD}\Ext^1_A(P_\Gamma/\rad^r P_\Gamma,L) @>>> \Ext^1_A(\rad^{r-1}P_\Gamma/\rad^rP_\Gamma,L)\end{CD}$
is an injection.

(c)  Continuing (b), if
$\dim\Ext^1_A(P_\Gamma/\rad^rP_\Gamma,L)=\dim\Ext^1_{\gr A}(\gr(P_\Gamma/\rad^rP_\Gamma),L)$, then
the restriction map $\Ext^1_{\gr A}(\gr(P_\Gamma/\rad^rP_\Gamma),L)\to\Ext^1_{\gr A}(\rad^{r-1}P_\Gamma/\rad^rP_\Gamma,L)$
is an injection.\end{prop}

\begin{proof} We first prove (a). The commutative diagram
 $$\begin{CD} 0 @>>> \rad^{r-1}P @>>> P @>>> P/\rad^{r-1}P @>>> 0\\
@. @VVV @VVV @VVV @. \\
0 @>>> \rad^{r-1}P/\rad^rP @>>> P/\rad^rP @>>> P/\rad^{r-1}P @>>> 0 \end{CD}
$$
has exact rows, so the long exact sequence of $\Ext^\bullet_A$ provides a commutative diagram
$$\begin{CD}  \Hom_A(\rad^{r-1}P,L) @<\sim<< \Hom_A(\rad^{r-1}P/\rad^rP,L)\\
@VVV   @VV{\alpha}V \\
\Ext^1_A(P/\rad^{r-1}P,L) @= \Ext^1_A(P/\rad^{r-1}P,L)\\
@VVV @VVV\\
\Ext^1_A(P,L) @<<< \Ext^1_A(P/\rad^rP,L)\\
@. @VV{\beta}V \\
@. \Ext^1_A(\rad^{r-1}P/\rad^rP,L)\end{CD}
$$
 in which the columns are exact. To show that $\beta$ is injective, we show that
$\alpha$ is surjective. However, because $L$ is completely reducible, the top horizontal map
is an isomorphism. Since $P$ is projective, $\Ext^1_A(P,L)=0$. The result follows by a diagram
chase.

Now (b) follows from (a), using (\ref{derivedembedding}) and the fact that if $P\in\Amod$ is a projective, then $P_\Gamma:=j^*P\in\Amod[\Gamma]$ is
projective.

To prove (c), write $P_\Gamma^{0,r}$ for $P_\Gamma/\rad^rP_\Gamma$ and $P_\Gamma^{r-1,r}$ for $\rad^{r-1}P_\Gamma/
\rad^rP_\Gamma$. Without loss of generality, we can assume that $L$ is irreducible. Let $\Omega$ be the kernel of the surjection
$P\twoheadrightarrow P_\Gamma^{0,r}$.

\noindent{\underline{Claim.}} No composition factor $L(\nu)$ of $\Omega/(\rad\Omega+\rad^rP)$
satisfies $\nu\in\Gamma$.  Otherwise, there is a nonzero homomorphism $\Omega\to L(\nu)$ with kernel $R\supseteq\rad\Omega+\rad^rP$
and with $\nu\in\Gamma$.
The composition factors of $M:=P/R$ are just those of $P_\Gamma^{0,r}=P/\Omega$ and $L(\nu)$. Thus, all composition factors
of $M$ have highest weight in $\Gamma$. Since $\head M=\head P=\head P_\Gamma$, $M$ is a homomorphic image of $P_\Gamma$.
Since $M$ is a homomorphic image of $P/\rad^rP$ by construction, its radical length is at most $r$. Hence, $M$ is
a homomorphic image of $P_\Gamma/\rad^rP_\Gamma=P_\Gamma^{0,r}$. However, $\dim M >\dim P_\Gamma^{0,r}$, contradicting
the existence of $L(\nu)$, and proving the claim.

Now put $Y=\rad\Omega+\rad^rP$, and write $\Omega/\rad\Omega=X/\rad\Omega\oplus Y/\rad\Omega$, a direct sum of (completely reducible)
$A$-submodules. for some $A$-submodule $X$ of $\Omega$ containing $\rad\Omega$. By the previous paragraph, $L$ is not a composition factor of $X/\rad\Omega$. Let $a=\dim\Ext^1_A(P_\Gamma^{0,r},L)$. Obviously, $\Omega/\rad\Omega$, and thus $\Omega/X$ has a quotient
module $\Omega/Z$ isomorphic to $L^{\oplus a}$ (since $\dim\Hom_A(\Omega,L)=\dim\Ext^1_A(P_\Gamma^{0,r},L)=a$). Put
$E=P/Z$. This gives an exact sequence $0\to L^{\oplus a}\to E\to P_\Gamma^{0,r}\to 0$. Notice
$L^{\oplus a}\cong\Omega/Z=(X+\rad^rP)/Z= (Z+\rad^rP)/Z=(Z+ (\rad^rA)P)/Z\cong \rad^rA(P/Z)$
identifies with $\rad^r E$.  Passage to $\gr E$ gives an exact sequence $0\to (L^{\oplus a})(r)\to \gr E\to
\gr P_\Gamma^{0,r}\to 0$ of graded $\gr A$-modules. Regarding $\gr E$ as an ungraded extension, the push-out $\rho_\pi$ via any projection $\pi:L^{\oplus a}\to L$ is non-split. (It clearly arises from a graded push-out in which $L^{\oplus a}$ and $L$ have degree $r$. Observe that $\gr E$ and, thus, its graded push-out---a homomorphic image---are generated in degree
$0$.) Let $\pi_1,\cdots,\pi_a$ be the standard projections $L^{\oplus a}\to L$, and put $\rho_i=\rho_{\pi_i}$. Let
$c_1,\cdots, c_a\in K$ and let $\pi=\sum c_i\pi_i:L^{\oplus a}\to L$. All maps $L^{\oplus a}\to L$ are obtained in this
way, and $\pi\not=0$ if and only if some $c_i\not=0$. The push-out $\rho_\pi$ is $\sum c_i\rho_{\pi_i}=\sum c_i\rho_i$.
In particular, $\sum c_i\rho=0$ if and only if $\rho_\pi=0$, which occurs if and only if $\pi=0$, or, equivalently,
all $c_i=0$. That is, $\rho_1,\cdots, \rho_a$ are linearly independent in $\Ext^1_{\gr A}(\gr P_\Gamma^{0,r},L)$. By
dimension considerations, these elements form a basis for $\Ext^1_{\gr A}(\gr P_\Gamma^{0,r},L)$.
  Also, if $\pi\not=0$, then $\rho_\pi$ remains non-split upon pull-back through the map $P_\Gamma^{r-1,r}\hookrightarrow (\gr P_\Gamma^{0,r})_{r-1}\subseteq P_\Gamma^{0,r}$.
(Observe that the $\gr A$-module $(\gr E)_{r-1}\oplus (\gr E)_r$ is generated in
 degree $r-1$.) This pull-back is the image of $\rho_\pi$ under the restriction map $\Ext^1_{\gr A}(P_\Gamma^{0,r}\to\Ext^1_{\gr A}(P_\Gamma^{r-1,r},L)$.
 This proves (c).
  \end{proof}

\begin{rem} After establishing the claim in the proof of (c), the proof may
be concluded by using an appropriate choice of the $\gr$-construction. Indeed, the implicit correspondence
of push-outs of $\gr E$ and $E$ in the proof as given can be seen to arise from a $\gr$-construction. The $\gr$ construction behaves here similarly as in the context of
    Theorem \ref{grThm1} and Remarks \ref{grRem2}(b,c). Note also that, by the injectivity (\ref{injection}) of the $\gr$-construction, the equality of dimensions assumed in the hypothesis of (c) is equivalent                                             to an isomorphism via the $\gr$-construction. In practice, this is checked
    using Theorem \ref{movinglemma}(b).
\end{rem}

Suppose that $A=\bigoplus_{n\geq 0}A_n$ is a graded QHA with weight poset $\Lambda$; see \cite{CPS1a}, \cite{CPS1}, \cite{CPS2},  \cite{CPS5}. For each $\lambda\in\Lambda$, $L(\lambda)\in\Agrmod$ denotes the
irreducible $A$-module $L(\lambda)$ viewed as a graded $A$-module concentrated in grade 0.
Also, $\Delta(\lambda),\nabla(\lambda),P(\lambda),Q(\lambda)\in\Agrmod$. We can assume that $\Delta(\lambda)$
and $P(\lambda)$ are generated in grade 0, and that, dually,  $\soc\nabla(\lambda)\subseteq\nabla(\lambda)_0$ and $\soc Q(\lambda)\subseteq Q(\lambda)_0$.
Furthermore, $P(\lambda)$ has a filtration by graded submodules with sections of the form $\Delta(\nu)(s)$, $s\geq 0$. (Thus, the ``top" section of
$P(\lambda)$ is $\Delta(\lambda)$ in $\Agrmod$.) Of course, $P(\lambda)\to\Delta(\lambda)$ is the projective cover in $\Agrmod$. Similar statements
hold for $Q(\lambda)$.

 \begin{prop}\label{gradedfiltrations} Suppose $A$ is a graded QHA over a field $K$. Let $Q$ be a finite dimensional graded $A$-module which has
a filtration by graded submodules with sections  $\nabla(\tau)(s)$, $\tau\in\Lambda$, $-s\in\mathbb N$. For
$\tau\in\Lambda$, $-s\in\mathbb Z$, the number $[Q:\nabla(\lambda)(s)]$  of occurrences of $\nabla(\tau)(s)$ as a section in such a filtration
is given by
$[Q:\nabla(\tau)(s)]=\dim\grHom_{A}(\Delta(\lambda)(s),Q(\mu)).$ \end{prop}

 By (\ref{elemordering}), $\dim\,\grExt^n_{A}(\Delta(\lambda)(r),\nabla(\mu)(s))=\delta_{n,0}\delta_{\lambda,\mu}\delta_{r,s},$
 which implies the stated result.
If $Q=Q(\mu)$ is the injective envelope in $\Agrmod$ of $L(\mu)$, then
$ [Q(\mu):\nabla(\tau)(s)]=[\Delta(\tau)_{-s}:L(\mu)]$,
a graded form of  Brauer-Humphreys reciprocity; see \cite[Prop. 1.2.4(a)]{CPS2}.

Let $A$ be an arbitrary QHA. As defined in \cite[\S2]{CPS5}, the HWC $\Amod$ (or just $A$)
satisfies (by definition) the (SKL$'$) condition\footnote{A stronger condition, in which $L(\mu)$ is replaced by $\nabla(\mu)$ in the
 first line and by $\Delta(\mu)$ in the second line of (SKL$'$) is called the strong Kazhdan-Lusztig property (SKL). It can be shown that (SKL) $\implies$ (SKL$'$), and that
 (SKL) is equivalent to (SKL$'$) holding for all algebras $eAe$, with $e$ an idempotent associated to a poset coideal in $\Lambda$; see \cite[\S2.4]{CPS5}. Unfortunately,
 not all these idempotents may be chosen from the subalgebra $\fa$ we will consider later, and so we work with (SKL$'$) only, in dealing with
 $A$. The property (SKL) will hold for $\gr A$ once (SKL$'$) is established for $A$ \cite[Thm. 2.4.2]{CPS5}.}
 with respect to a function $l:\Lambda\to\mathbb Z$, provided, that for $\lambda,\mu\in\Lambda$,
$n,i\in{\mathbb N}$,
$$
\begin{cases}(1)\,\,\,\Ext^n_A(\Delta^i(\lambda),L(\mu))\not=0\implies \,n\equiv l(\lambda)-l(\mu)+i; \,{\text{\rm and}} \\
(2)\,\,\,\Ext^n_A(L(\mu),\nabla_{-i}(\lambda))\not=0\implies n\equiv l(\lambda)- l(\mu)+i.\end{cases}
\leqno({\text{\rm SKL}}')$$
Here $\Delta^i(\lambda):=\rad^i\Delta(\lambda)$ and $\nabla_{-i}(\lambda):=\nabla(\lambda)/\soc_{-i}\nabla(\lambda)$.
If $\Gamma$ is a poset ideal in $\Lambda$,  (\ref{derivedembedding}) implies that
if $\Amod$ satisfies (SKL$'$), then  $A/J$-mod also satisfies (SKL$'$) with respect
to $l|_\Gamma$. Indeed, it is only necessary that condition (SKL$'$) holds for $\lambda,\mu\in\Gamma$.
Similar remarks hold for the (KL) property discussed in the next paragraph.

If condition (SKL$'$) holds for  $i=0$, then $\Amod$ is said to satisfy the Kazhdan-Lusztig property (KL) (with respect to $l:\Lambda\to\mathbb Z$). More precisely, the (KL) property holds for $\Amod$ (with respect to $l:\Lambda\to\mathbb Z$)
provided that, for $\lambda,\mu\in\Lambda$, $n\in{\mathbb N}$,
$$
\begin{cases}(1)\,\,\,\Ext^n_A(\Delta(\lambda),L(\mu))\not=0\implies \,n\equiv l(\lambda)-l(\mu); \,{\text{\rm and}} \\
(2)\,\,\,\Ext^n_A(L(\mu),\nabla(\lambda))\not=0\implies n\equiv l(\lambda)- l(\mu).\end{cases}
\leqno({\text{\rm KL}})$$
 The property
(KL) implies surjectivity of the natural map
$\Ext^n_A(L(\lambda),L(\mu))\to\Ext^n_A(\Delta(\lambda),L(\mu))$ $\forall n\in{\mathbb N},$ $\forall \lambda,\mu\in\Lambda$;
 see \cite[Thm. 4.3]{CPS2}. This means that
 $\Ext^n_A(\Delta(\lambda),L(\mu))\to\Ext^n_A(\rad\Delta(\lambda),L(\mu))$ is the zero map 
 $\forall n\in{\mathbb N},  \forall \lambda,\mu\in\Lambda$.
 Thus,
\begin{equation}\label{afterSKL}\Ext^n_A(\Delta(\lambda),L(\mu))\to\Ext^n_A(\Delta^r(\lambda),L(\mu))\quad{\text{\rm is the zero map}},\quad\forall r>0,\forall n\in{\mathbb N},
 \forall \lambda,\mu\in\Lambda.\end{equation}

A graded QHA $A$ (or its graded module category) is said to have a graded Kazhdan-Lusztig
theory, or satisfy the graded Kazhdan-Lusztig property, with respect to $l:\Lambda\to\mathbb Z$, provided the following holds: Given $\lambda,\mu\in\Lambda$, if either $\grExt^n_{A}(\Delta(\lambda),L(\mu)(m))\not=0$ or
$\grExt^n_{A}(L(\mu),\nabla(\lambda)(m))\not=0$, then $m=n\equiv l(\lambda)- l(\mu)$ mod$\,2$. (Equivalently, given property (KL), the standard modules are linear and the costandard modules have a dual property.) 
In this case, $A$ is a Koszul
algebra; see \cite{CPS1}, \cite{CPS2}, \cite{CPS5}.  

Given $\lambda,\nu\in\Lambda$, the  Kazhdan-Lusztig polynomial $P_{\nu,\lambda}$ is defined
as\footnote{Strictly speaking, these are the left Kazhdan-Lusztig polynomial. The analogous right Kazhdan-Lusztig
polynomial $P_{\nu,\lambda}^R$ is defined by the same formula, but replacing $\dim\Ext^n_A(L(\lambda),\Delta(\nu))$
by $\dim\Ext^1_A(\Delta(\nu),L(\lambda))$. However, in all the cases we are interested in the category $\Amod$ has a
duality, so that $P_{\nu,\lambda}=P^R_{\nu,\lambda}$.}
\begin{equation}\label{polynomial}
P_{\nu,\lambda}= t^{l(\lambda)-l(\nu)}\sum_{n=0}^\infty \dim\Ext^n_A(L(\lambda),\Delta(\nu))t^{-n}\in{\mathbb Z}[t,t^{-1}].\end{equation}
In  practice (see \S8), $P_{\nu,\lambda}\in{\mathbb Z}[t^2]$ identifies with the Kazhdan-Lusztig polynomial of a Coxeter group.

\begin{thm} (\cite[Thm. 2.2.1, Rem. 2.2.2(b)]{CPS5})\label{SKLtheorem} If $A$ is a QHA such that ({\text{\rm SKL}}$'$) holds relative to $l:\Lambda\to\mathbb Z$, then $\gr A$ is a graded QHA with poset $\Lambda$. Also, the graded Kazhdan-Lusztig property holds
for $\gr A$, so that  $\gr A$ is a
Koszul algebra. In addition, $A$-mod and $\gr A$-mod have the same Kazhdan-Lusztig polynomials. Finally, the
standard (resp., costandard) module in $\gr A$--mod indexed by $\lambda\in \Lambda$ is $\gr\Delta(\lambda)$ (resp., $\gr\nabla(\lambda)$).\end{thm}

Finally, we end this section with the following important remark.
\begin{rem}\label{LCF} Let $A$ be a QHA with weight poset $\Lambda$, and let $l:\Lambda\to
\mathbb Z$ be given.  It is important to know if the highest weight category $\Amod$ satisfies the (KL)
property with respect to $l$ stated above. In fact, considerable effort has been devoted to conditions
which are equivalent the validity of the (KL) property. Some of these are laid out in \cite[\S5]{CPS1}. 
One of these equivalent conditions is the validity of the Lusztig character formula LCF (suitably formulated).  
Rather than repeat this material (see, in particular, \cite[Thm. 5.3 \& Cor. 5.4]{CPS1}), we mention, looking
ahead to \S10 below (and using the notation there), that the LCF holds for $p$-regular weights in the
Janzten region $\Jan$ for a semisimple simple algebraic group $G$ (over an algebraically closed field
of positive characteristic $p\geq h$) provided, given $\lambda=w\cdot\lambda^-\in \Jan$, $\lambda^-\in C^-$,
the character equality
$$\ch L(\lambda)=\sum_{y\leq x\in W_p} (-1)^{l(w)-l(y)}P_{y,w}(-1)\ch \Delta(y\cdot \lambda^-)$$
is valid. See \cite[\S5]{CPS1} for more details. Similar comments hold for $U_\zeta$-mod discussed
in \S8.
\end{rem}

\section{Hopf algebras and further results on extensions} Let  $U$ be a Hopf algebra over a field $K$ having a normal
Hopf subalgebra $u$. Assume that the antipodes on $U$ and $u$ are surjective.\footnote{This assumption is automatic for  the
  Hopf algebras we consider. } For $U$-modules $M,N$,
$\Hom_K(M,N)$ has a natural $U$-module structure given by  $(u\cdot f)(m)=\sum u_{(1)}f(S(u_{(2)}m)$ for $f\in\Hom_K(M,N)$, $m\in M$, $u\in U$.
(Here $S:U\to U$ is the antipode and $\Delta(u):=\sum u_{(1)}\otimes u_{(2)}$, the image of $u$ under the comultiplication $\Delta:U\to U\otimes U$.) The surjectivity of the
antipode $S$
implies that $\Hom_U(M,N)=\Hom_K(M,N)^U$; see \cite[Prop. 2.9]{APW1}.
Let $U^\dagger:=U//u$, the quotient of $U$ by the Hopf ideal $u_+U=Uu_+$ generated by the augmentation
ideal $u_+$ of $u$. We are given finite dimensional augmented algebras $A$ and $\fa$ over $K$, with $\fa$ a normal (augmented) subalgebra
of $A$.  Assume there are compatible surjective
algebra homomorphisms
\begin{equation}\label{compatable}\renewcommand{\arraystretch}{1.5}\begin{array}[c]{cccc}
u &\twoheadrightarrow &\fa\\
\downarrow &&\downarrow\\
U &\twoheadrightarrow & A\end{array}
\end{equation}
 in which the vertical arrows are inclusions.
 We also assume that
\begin{equation}\label{radicalassumption}
\rad A=(\rad\fa)A.
\end{equation}
Thus, given an $A$-module $L$, then $L$ is completely reducible if and only if $L|_\fa$ is completely reducible.

Then
$\Hom_A(M,N)=\Hom_U(M,N)=\Hom(M,N)^U = \Hom_u(M,N)^{U^\dagger}=\Hom_{\fa}(M,N)^{U^\dagger},$
for $M,N\in A$-mod.
  When $N=L$ is a completely reducible, versions of this identity exist for $\gr A$ and $\gr\,\fa$.
Thus, if $M$ be a $\gr A$-module, then $M/\rad M$ is a $\gr A/(\rad\gr A)\cong A/\rad A$-module, and
 $$\begin{aligned}\Hom_{\gr A}(M,L) &\cong \Hom_{\gr A}(M/\rad M,L) \\
&\cong\Hom_{A}(M/\rad M,L) \\
  & \cong\Hom_{\fa}(M/\rad M,L)^{U^\dagger} \\
&\cong\Hom_{\gr \fa}(M/\rad M,L)^{U^\dagger} \\
& \cong\Hom_{\gr\fa}(M,L)^{U^\dagger}.\end{aligned}$$
The intermediate isomorphisms explain the isomorphism $\Hom_{\gr A}(M,L)\cong\Hom_{\gr\fa}(M,L)^{U^\dagger}$.
Finally, if $M$ is a graded $\gr A$-module and $L$ is a graded completely reducible module,
$$\begin{aligned} \grHom _{\gr A}(M,L) &\cong\grHom _{\gr A}(M/\rad M,L)\\ &\cong\bigoplus_{r\in\mathbb Z}\Hom_{\gr A}((M/\rad M)_r,L_r)\\
&\cong\bigoplus_{r\in\mathbb Z}\Hom_{A}((M/\rad M)_r,L_r)\\
&\cong\bigoplus_{r\in\mathbb Z}\Hom_{\fa}((M/\rad M)_r,L_r)^{U^\dagger}\\
&\cong\grHom _{\gr \fa}(M/\rad M,L)^{U^\dagger}\\
&\cong\grHom _{\gr\fa}(M,L)^{U^\dagger}.\end{aligned}
$$

\begin{prop}\label{degenerateProp1} Assume the setup of (\ref{compatable}) and (\ref{radicalassumption}).
Let $L$ be a completely reducible $A$-module. Let $M$ be an $A$-module satisfying the following
property (*): there is a projective cover $P\overset\nu\longrightarrow M$
in which $P$ is a projective $\fa$-module. Then there are isomorphisms
$\Ext^1_A(M,L)\cong\Ext^1_\fa(M,L)^{U^\dagger}$ and $\Ext^1_{\gr A}(\gr M,L)\cong\Ext^1_{\gr\fa}(\gr M,L)^{U^\dagger}$,
which are natural with respect to morphisms between $A$-modules satisfying (*).
Finally, if $L$ is given a graded structure,
$\grExt^1_{\gr A}(\gr M,L)\cong\grExt^1_{\gr\fa}(\gr M,L)^{U^\dagger},$
which are also natural (giving graded homomorphisms) with respect to  morphisms
in $\Amod$ satisfying (*). (In particular, there are natural actions of $U^\dagger$ on all
the indicated $\Ext$-groups above.)
\end{prop}

\begin{proof} There exist  exact sequences
$$\begin{cases}
0\to \Omega\to P\overset\nu\to M\to 0,\\
0\to\gr^\#\Omega\to\gr P\overset{\gr\nu}\longrightarrow\gr M\to 0,\end{cases}$$
 resulting in isomorphisms
$$\begin{cases}\Ext^1_A(M,L)\cong\Hom_A(\Omega,L),\\
\Ext^1_{\gr A}(\gr M,L)\cong\Hom_{\gr A}(\gr^\#\Omega,L),\\
\Ext^1_\fa(M,L)\cong
\Hom_\fa(\Omega,L),\\ \Ext^1_{\gr \fa}(\gr M,L)\cong\Hom_{\gr\fa}(\gr^\#\Omega,L).\end{cases}$$
 In fact, by (\ref{radicalassumption}), $P\twoheadrightarrow M$ (resp., $\gr P\twoheadrightarrow
\gr M$) is a projective cover of $M$ (resp., $\gr M$) in $\famod$ (resp., $\gr\famod$).
Also, if $L$ has a graded structure, then
$\grExt^1_{\gr A}(\gr M,L)\cong\grHom _{\gr A}(\gr^\#\Omega,L)$ and
$\grExt^1_{\gr \fa}(\gr M,L)\cong\grHom _{\gr\fa}(\gr^\#\Omega,L)$.
We have previously shown that
\begin{itemize}
\item[ (i)]
$\Hom_{\fa}(\Omega,L)$ has a natural $U^\dagger$-structure, and
$\Hom_{A}(\Omega,L)\cong\Hom_{\fa}(\Omega,L)^{U^\dagger};$
\item[(ii)] $\Hom_{\gr \fa}(\gr^\#\Omega,L)$ has a natural $U^\dagger$-structure,
and $\Hom_{\gr A}(\gr^\#\Omega,L)\cong\Hom_{\gr \fa}(\gr^\#\Omega,L)^{U^\dagger};$
\item[(iii)]  $\grHom _{\gr \fa}(\gr^\#\Omega,L)$ has a natural
$U^\dagger$-structure, and
$\grHom _{\gr A}(\gr^\#\Omega,L)\cong\grHom _{\gr\fa}(\gr^\#\Omega,L)^{U^\dagger}.$
\end{itemize}
These isomorphisms are natural in the category of modules satisfying (*), and can be used to
define the $U^\dagger$-action and the isomorphisms in the statement of the proposition. \end{proof}

\section{Quasi-hereditary structures on $\gr B$}

The algebra $B$ of the title of this section will be a introduced above Theorem \ref{grBisQHA}. 
It will be a QHA quotient of a QHA $A$. The algebra $A$ will be part of a pair $(A,\fa)$ satisfying the following conditions.

\begin{hyp}\label{hypothesisofsection6} 
\begin{enumerate}
\item[(1)] $A$ is a split QHA over the field $K$ with poset $\Lambda$. (See \S4)
\item[(2)] $\fa=\bigoplus_{n\geq 0}\fa_n$ is a tightly graded subalgebra of $A$ such that the pair $(A,\fa)$ satisfies the setup given in (\ref{compatable})
for a pair $(U,u)$ of Hopf algebras. In particular, $\fa$ is an augmented, normal subalgebra of $A$.
\item[(3)] $(\rad\fa)A=\rad A$ as in (\ref{radicalassumption}). In particular, $\rad A\subseteq \fa_+A$, so $A/\fa_+A$ is semisimple.
\item[(4)] $\fa_0\subseteq A_0$, a Wedderburn complement of $A$.
\end{enumerate}
\end{hyp}

Condition (4) is automatic when $\fa_0$ and $A/\rad A$ are separable algebras over $K$. It will hold
 when $\fa_0$ and $A/\rad A$ are split semisimple, or in case
$K$ has characteristic 0 or is algebraically closed. (These will be the cases in \S8.)  In fact, the separability of $A/\rad A$ guarantees that $A$
has a subalgebra (Wedderburn complement) $A'_0$ mapping isomorphically onto $A/\rad A$ under the quotient map $\pi:A\twoheadrightarrow A/\rad A$.
If $\fa_0'$ is the subalgebra of $A'_0$ mapping isomorphically to $\pi(\fa_0)$ by $\pi$, then $\fa_0'$ and $\fa_0$ are
both Wedderburn complements for $E:=\pi^{-1}(\pi(\fa_0))$. By the Wedderburn-Malcev theorem, there exists $x\in\rad A=\rad E$
such that $\fa_0=(1-x)\fa_0'(1-x)^{-1}.$ Then $$\fa_0\subseteq A_0:=(1-x)A'_0(1-x)^{-1},$$
as required.

An $(A,A_0)$-module is defined to be a pair $(M,M')$ with $M,M'$ modules for $A,A_0$, respectively, and with $M'\subseteq M|_{A_0}$.
In addition, the pair $(M,M')$ is defined to have a (non-negative) $\fa$-graded structure if $M|_\fa$ is a non-negatively $\fa$-graded module and if $M'=M_0$ is its grade 0-term.

\begin{lem}\label{infinitesimallemma1} Suppose $(M,M_0)$ is a $\fa$-graded $(A,A_0)$-module such that $M_i=0$ for $i\geq r$ for
a given integer $r>0$. Let $L$ be a completely
reducible $A$-module, viewed as a homogeneous $\fa$-module concentrated in grade 0.  Let $E$ be an $A$-extension of $M$ by $L$ which is also an $\fa$-graded extension of
$M$ by $L(r)$. Then the grade 0 term $E_0$ of the $\fa$-grading may be rechosen to be an $A_0$-submodule of $E$, with $E_1,\cdots, E_r$ remaining the same, and
$E_0$ having image $M_0$ in $M$. (Thus, $(E,E_0)$ is an $\fa$-graded $(A,A_0)$-module with $(M,M_0)$ as a natural homomorphic image.)\end{lem}

\begin{proof} Write  $E|_{\fa}=E_0\oplus E_1\oplus\cdots\oplus E_r$, the direct sum of its grades.  Since $M_0$ is $A_0$-stable, the extension $E$ of $M$ by $L$ defines an $A_0$-extension of $M_0$ by $L$. Since $A_0$ is semisimple, there
 is an $A_0$-submodule $E'_0$ of $E$  mapping isomorphically onto $M_0$. Since $\fa_iE_r=0$ for $i>0$, it follows that $\fa_iE'_0
\subseteq E_i$ for $i>0$. Therefore,
$E=E_0'\oplus E_1\oplus \cdots\oplus E_r$ defines a new grading of $E$ as an $\fa$-module. By construction, $(E,E'_0)$ is an $(A,A_0)$-module which is graded
as an $\fa$-module.
\end{proof}

Fix a proper poset ideal $\Theta$ of $\Lambda$, and let $\lambda\in\Theta$. To simplify notation in 
the remainder of this section, we define 
\begin{equation}\label{Deltaequation} \Delta= P_\Theta(\lambda),\end{equation}
the largest quotient of the projective cover $P(\lambda)$ of $L(\lambda)$ in $\Amod$ all of whose
composition factors $L(\theta)$ satisfy $\theta\in\Theta$. Thus, if $\Theta=\{\nu\in\Lambda\,|\,
\nu\leq\lambda\}$, then $\Delta=\Delta(\lambda)$, which motivates the notation.

 In applying the  lemma below to Theorem \ref{MainInfinitesimalThm}, we will assume, inductively, that
$\Delta^{0,r}|_\fa$ has a graded $\fa$-module structure with the $0$-grade term $(\Delta^{0,r})_0$ stable
under the action of $A_0$. In particular,  Theorem \ref{movinglemma}(a) allows Proposition \ref{grProp1} and Proposition \ref{grProp2}, with $M=
\Delta^{0,r}$, to be used simultaneously provided the $A$-projective cover of $\Delta^{0,r}$  (which is also the projective cover of
$\Delta$ and $L(\lambda)$) is $\fa$-projective.

\begin{lem}\label{bigdiagramlemma} Let $\Delta\in\Amod$ be as in (\ref{Deltaequation}) for the fixed
$\lambda\in\Theta$. Let $r\geq 1$. Let $L\in\Amod$ be completely reducible with summands $L(\gamma)$, $\gamma\in\Theta$.
 Assume, for each $\theta\in\Theta$, the 
projective cover $P(\theta)$ of $L(\theta)$ in $\Amod$ is  projective for $\fa$. Finally, assume
 $\Delta^{0,r}|_\fa\in\fagrmod_0$ so that $\Delta^{0,r}_0$ is $A_0$-stable. Then there is a commutative diagram

\begin{equation}\label{bigdiagram}
\begin{array}{ccccc}
\Ext^1_A(\Delta^{0,r},L) &
\bpict(30,0) \put(0,5){\vector(1,0){30}}\put(12,8){$\alpha$}\epict
& \Ext^1_\fa(\Delta^{0,r},L) &
\bpict(30,0) \put(30,5){\vector(-1,0){30}}\put(13,8){$\beta$}\epict
& \grExt^1_\fa(\Delta^{0,r},L(r)) \\[6mm]
& \bpict(30,0) \put(23,-5){\vector(-1,1){17}}
 \put(17,6){$f$}\epict &
\bpict(30,0) \put(20,-5){\vector(0,1){20}}\put(23,3){$\rho$}\epict
&&
\bpict(30,0) \put(20,-5){\vector(0,1){20}}\put(23,3){$\sigma$}\epict
\\[6mm]
\Ext^1_{\rm{gr}A}(\rm{gr}\Delta^{0,r},$L$)  &
\bpict(30,0) \put(30,5){\vector(-1,0){30}}\put(13,8){$\delta$}\put(11,-2){$\sim$}\epict
& \Ext^1_{\rm{gr}\fa}(\rm{gr}\Delta^{0,r},$L$)^{U^\dagger} &
\bpict(30,0) \put(30,5){\vector(-1,0){30}}\put(13,8){$\epsilon$}\epict
& \rm{}\grExt^1_{\rm{gr}\fa}(\rm{gr}\Delta^{0,r},$L(r)$)^{U^\dagger} \\[6mm]
\bpict(30,0) \put(20,15){\vector(0,-1){20}}\put(23,3){$\xi$}\epict
&&
\bpict(30,0) \put(20,15){\vector(0,-1){20}}\put(23,3){$\tau$}\epict
& (I) &
\bpict(30,0) \put(20,15){\vector(0,-1){20}}\put(23,3){$\phi$}\epict
 \\[6mm]
\Ext^1_{\rm{gr} A}(\Delta^{r-1,r},$L$)&
\bpict(30,0) \put(30,5){\vector(-1,0){30}}\put(13,8){$\kappa$}\put(11,-2){$\sim$}\epict
&\Ext^1_{\rm{gr}\fa}(\Delta^{r-1,r},$L$)^{U^\dagger}&
\bpict(30,0) \put(30,5){\vector(-1,0){30}}\put(12,9){$\zeta$}\put(11,-2){$\sim$}\epict
&\rm{}\grExt^1_{\rm{gr}\fa}(\Delta^{r-1,r},$L(r)$)^{U^\dagger}
\\[6mm]
&&
\bpict(30,0) \put(20,15){\vector(0,-1){20}}\put(23,3){$\pi$}\epict
 & (II) &
\bpict(30,0) \put(20,15){\vector(0,-1){20}}\put(23,3){$\theta$}\epict
\\[6mm]
&& \Ext^2_{\rm{gr}\fa}(\rm{gr}\Delta^{0,r-1},$L$) &\bpict(30,0) \put(30,5){\vector(-1,0){30}}\put(12,9){$\iota$}
\epict
& \rm{}\grExt^2_{\rm{gr}\fa}(\rm{gr}\Delta^{0,r-1},$L(r)$)
\end{array}
\end{equation}
in which $\xi,\iota$ are injective, the maps $\zeta,\kappa,\delta, f$ are isomorphisms, and $\ker\pi=\Image{\tau}$ (resp.,
$\ker\theta=\Image{\phi}$). \end{lem}

\begin{proof} We begin by establishing the existence of the map $f$ and by defining
the maps $\delta,\epsilon,\rho,\sigma$.
In order to obtain the top two rows of the diagram, consider a smaller diagram
$$\begin{CD} \Ext^1_A(\Delta^{0,r},L) @>{\mu}>> \Ext^1_{\gr A}(\gr\Delta^{0,r},L)\\
@VV{\alpha}V @VV{\delta'}V\\
\Ext^1_\fa(\Delta^{0,r},L) @>{\nu}>> \Ext^1_{\gr\fa}(\gr\Delta^{0,r},L) \\
@AA{\beta}A @AA{\epsilon'}A\\
\grExt^1_\fa(\Delta^{0,r},L(r)) @>{\omega}>> \grExt^1_{\gr \fa}(\gr\Delta^{0,r},L(r))
\end{CD}$$
The three horizontal maps are all induced by the $\gr$-construction, and the vertical maps are the ``obvious" ones. Commutativity is a
consequence of Propositions \ref{grProp1} and Proposition \ref{grProp2}.
A  projective cover of $\Delta^{0,r}$ is also a projective cover of $L(\lambda)$, hence is $\fa$-projective by hypothesis. The maps $\mu,\nu,\omega$ are isomorphisms, by Proposition
 \ref{grProp2} and Theorem
\ref{movinglemma}(b) (applied to $M=\Delta^{0,r}$). Now $\rho$ (resp., $\sigma$, $\epsilon$) is obtained by restricting $\nu^{-1}$ (resp., $\omega^{-1}$, $\epsilon'$) to the
submodule of $U^\dagger$-fixed points. We have defined the upper right-hand square and the top row of (\ref{bigdiagram}).
We now consider the construction of the diagonal map $f$.
 Proposition \ref{degenerateProp1} and its argument show that the images of $\alpha$ and $\delta'$ are $\Ext^1_\fa(\Delta^{0,r},L)^{U^\dagger}$ and $\Ext^1_{\gr\fa}(\gr\Delta^{0,r},L)^{U^\dagger}$, respectively. Both $\alpha$ and $\delta'$ are injective, giving isomorphisms to their
images. Let
$$\delta:\Ext^1_{\gr\fa}(\gr\Delta^{0,r},L)^{U^\dagger}\to\Ext^1_{\gr A}(\gr\Delta^{0,r},L)$$
be the map inverse to $\delta'$ on its image. Define
$$f:\Ext^1_{\gr\fa}(\gr\Delta^{0,r},L)^{U^\dagger}\to\Ext^1_A(\Delta^{0,r},L)$$ to be the composition of $\delta$ with $\mu^{-1}$.

Our construction shows that $\rho=\alpha\circ f$ and $\rho\circ\epsilon=\beta\circ\sigma$, and that $f$ is an isomorphism.
 This establishes
all the claimed interrelationships of the maps in the top two rows of the four row diagram, and the claimed properties of these maps.
The isomorphism $\kappa$ is obtained using Proposition \ref{degenerateProp1}, and  the commutativity of the square to which it belongs (along with the
definitions of $\xi$ and $\tau$) are obvious.

Boxes (I) and (II) {\it without} $(-)^{U^\dagger}$ are commutative as evident parts of long exact sequences of $\Ext$ and $\grExt$, with the maps
$\epsilon,\zeta,\iota$ from right to left standard injections provided by (\ref{Extgrouprelationstograded}).  Also, $\epsilon$ and $\zeta$ have
factorizations fitting into a commutative diagram:
$$\begin{CD} \Ext^1_{\gr\fa}(\gr\Delta^{0,r},L)^{U^\dagger} @<{\supseteq}<< \grExt^1_{\gr A}(\gr\Delta^{0,r},L(r)) @<{\sim}<<
\grExt^1_{\gr\fa}(\gr\Delta^{0,r},L(r))^{U^\dagger} \\
@VVV  @VVV @VVV \\
\Ext^1_{\gr\fa}(\Delta^{r-1,r},L)^{U^\dagger} @<{\supseteq}<< \grExt^1_{\gr A}(\Delta^{r-1,r},L(r)) @<{\sim}<<\grExt^1_{\gr\fa}(\Delta^{r-1,r},L(r))^{U^\dagger}
\end{CD}
$$
in which the right-hand box is commutative by the naturality in Proposition \ref{degenerateProp1}. The commutativity of the left-hand box follows
from naturality in Proposition \ref{degenerateProp1} together with the obvious naturality of (\ref{Extgrouprelationstograded}). The fact that $\epsilon$ and
$\zeta$ are the composites of the top row and bottom row, respectively, is easily seen from the construction of these maps in terms of syzygies.  Note that box
(II) without the $(-)^{U^\dagger}$ gives maps whose restrictions to $\grExt^1_{\gr\fa}(\gr\Delta^{0,r},L(r))^{U^\dagger}$ and $ \grExt^1_{\gr\fa}(\Delta^{r-1,r},L(r))^{U^\dagger}$ are candidates for $\epsilon$ and $\zeta$, and it is really these restrictions that we have shown factor as above
with images in $\Ext^1_{\gr\fa}(\gr\Delta^{0,r},L)^{U^\dagger}$ and $\Ext^1_{\gr\fa}(\gr\Delta^{0,r},L)^{U^\dagger}$, respectively. Thus our argument establishes
the existence of box (II), including the $(-)^{U^\dagger}$.

The $A$-modules $\Delta$ and $\Delta^{0,r}$ have the projective cover $P=P(\lambda)$. It is
assumed that $P|_\fa$ is projective for $\fa$. Thus, the hypothesis of Theorem \ref{movinglemma}(b)
are satisfied with $M=\Delta^{0,r}$. Hence,  any choice of the $\gr$-construction gives an isomorphism
$$\Ext^1_A(\Delta^{0,r},L)\overset\sim\to\Ext^1_{\gr A}(\gr\Delta^{0,r},L),$$
 and so an gives equality of dimensions of these two spaces. Then  Proposition \ref{grProp3}(c) implies
that $\xi$ is injective as required.

The map $\zeta$ is an isomorphism, since the containment
$$\Ext^1_{\gr\fa}(\Delta^{r-1,r},L)^{U^\dagger}\supseteq\grExt^1_{\gr A}(\Delta^{r-1,r},L(r))$$ factors as
$$\Ext^1_{\gr\fa}(\Delta^{r-1,r},L)^{U^\dagger}\cong\Ext^1_{\gr A}(\Delta^{r-1,r},L)\supseteq\grExt^1_{\gr A}(\Delta^{r-1,r},L(r)).$$
Proposition \ref{degenerateProp1} has been used again for the isomorphism. The containment is an equality, since, for the tightly graded algebra $\gr A$, graded extensions between irreducible modules are split unless the grades differ by 1.
\end{proof}

\begin{thm}\label{MainInfinitesimalThm} Assume that Hypothesis \ref{hypothesisofsection6}  holds, and $\Delta=P_\Gamma(\lambda)$, where $\Gamma$ is a poset ideal in $\Lambda$ and $\lambda\in\Gamma$.
Assume that, for each $\gamma\in\Gamma$, the projective cover $P(\gamma)$ in $\Amod$ is projective as an $\fa$-module. Then $\Delta$ has an $\fa$-graded structure generated in grade 0 such that $\Delta_0$ is $A_0$-stable. In particular, these conclusions
hold if $\Delta=\Delta(\lambda)$, provided $P(\lambda)|_\fa$ is projective. \end{thm}

\begin{proof} We prove by induction on $r$ that $\Delta^{0,r+1}$ has an $\fa$-graded structure generated in grade 0 such that $\Delta^{0,r+1}_0$ is $A_0$-stable,
assuming that the similar result holds for $\Delta^{0,r}$. Let $L=\Delta^{r,r+1}$ in Lemmas \ref{infinitesimallemma1},  \ref{bigdiagramlemma}.
Let $\chi\in\Ext^1_A(\Delta^{0,r},L)$ be the extension corresponding to
$0\to\Delta^{r,r+1}\to \Delta^{0,r+1}\to\Delta^{0,r}\to 0.$
If  $\alpha(\chi)\in\Image{(\beta)}$, we can put an $\fa$-graded structure on $\Delta^{0,r+1}$, compatible with that
on $\Delta^{0,r}$, with $\Delta^{r,r+1}$ pure of grade $r$. Then Lemma \ref{infinitesimallemma1} applied with $M=\Delta^{0,r}$  completes the induction.
However, since $f$ is an isomorphism in (\ref{bigdiagram}) there is an element $\chi'\in \Ext^1_{\gr\fa}(\gr\Delta^{0,r},L)^{U^\dagger}$ with $f(\chi')=\chi$,
and so $\rho(\chi')=\alpha(\chi)$. Write $\tau(\chi')=\zeta(\chi^{\prime\prime})$, with $\chi^{\prime\prime}\in\grExt^1_{\gr\fa}(\Delta^{r-1,r},L(r))^{U^\dagger}$.
Then $\iota\theta(\chi^{\prime\prime})=\pi\tau(\chi')=0$. Thus, $\theta(\chi^{\prime\prime})=0$ since $\iota$ is injective. Hence, $\chi^{\prime\prime}=
\phi(\chi^{\prime\prime\prime})$ with $\chi^{\prime\prime\prime}\in\grExt^1_{\gr\fa}(\gr\Delta^{0,r},L(r))^{U^\dagger}$. Both $\epsilon(\chi^{\prime\prime\prime})$
and $\chi'$ have the same image under $\tau$ and hence are equal, using the injectivity of $\xi$ and the isomorphisms $\delta,\kappa$. Now we have
$\alpha(\chi)=\rho(\chi')=\rho(\epsilon(\chi^{\prime\prime\prime}))=\beta(\sigma(\chi^{\prime\prime\prime}))$, as desired.\end{proof}

We prove the main result of this section, establishing that, given a poset
ideal $\Gamma$ of $\Lambda$, if the hypotheses of Theorem \ref{MainInfinitesimalThm} holds for
$\Gamma$, then $\gr B$ is a QHA for  $B:=A_\Gamma$.

 \begin{thm}\label{grBisQHA} Assume that $A$ satisfies Hypothesis \ref{hypothesisofsection6}. 
 Let $\Gamma$ be a 
(non-empty) poset ideal in $\Lambda$ such that, given $\gamma\in\Gamma$,
the projective cover $P(\gamma)$ of $L(\gamma)$ in $\Amod$ is projective as an $\fa$-module. 
Let $B=A_\Gamma:=A/J$ be so that $\Bmod\cong\Amod[\Gamma]$ (for some ideal $J$ of $A$). Then 
$\gr B$ is a QHA with weight poset $\Gamma$. Also, the $\gr B$-modules $\gr \Delta(\gamma)$, $\gamma\in\Gamma$, are the standard
modules in the HWC $\gr B$-mod.\end{thm}

\begin{proof} Induction on $|\Gamma|$: If $|\Gamma|=1$, then $B\cong K$, and the
result is trivial. Thus, assume that the result holds for any proper ideal $\Theta$ of $\Gamma$.
In particular,  let $\tau\in\Gamma$ be maximal, and form the poset
ideal $\Theta:=\Gamma\backslash\{\tau\}$ of $\Lambda$.
By induction, $C:=B_\Theta$ is QHA with weight poset $\Theta$ and with standard modules $\gr\Delta(\theta)$, $\theta\in\Theta$. 

For $\gamma\in\Theta$, the projective cover $P(\gamma)$ of $L(\gamma)$ in $\Amod$ is, by 
hypothesis, a projective $\fa$-module. Thus, by Theorem \ref{MainInfinitesimalThm}, $P_\Theta(\gamma)
\in \fa$-grmod$_0$, and $P_\Theta(\gamma)_0$ is $A_0$-stable.  Therefore, the hypotheses of Theorem \ref{movinglemma} hold with $M=P_\Theta(\gamma)$ and $L=L(\mu)$ for any $\mu\in \Lambda$, so that
Theorem \ref{movinglemma}(b) gives an isomorphism $\Ext^1_A(P_\Theta(\gamma),L(\mu))\cong
\Ext^1_{\gr A}(\gr P_\Theta(\gamma),L(\mu))$. But, if $\mu\in\Theta$, then (\ref{derivedembedding}) implies that
$\Ext^1_A(P_\Theta(\gamma),L(\mu))\cong \Ext^1_C(P_\Theta(\gamma),L(\nu))=0$, because
$P_\Theta(\gamma)$ is projective in $C$--mod.  Therefore, $\Ext^1_{\gr A}(\gr P_\Theta(\gamma),L(\mu))=0$,
for all $\mu\in\Theta$. However,  there is an evident surjection
$\gr P(\gamma)\twoheadrightarrow \gr P_\Theta(\gamma)$ whose kernel has no quotient $L(\mu)$ with
$\mu\in\Theta$, by the $\Ext^1$-vanishing just proved. (Note that $\gr P(\gamma)$ is the projective cover
of $L(\gamma)$ in $\gr A$-mod, so it has a simple head.) 
It follows that $(\gr P(\gamma))_\Theta
\cong \gr P_\Theta(\gamma)$ as a $(gr A)_\Theta$-module. Therefore, letting $\gamma\in\Theta$ vary,
\begin{equation}\label{identity} \gr (A_\Theta)\cong(\gr A)_\Theta.\end{equation}. 
 
 Let $J=\ker(B\to B_\Theta)$.  (We sometimes denote $B_\Theta$ by $C$.)
 In the notation of \S2, $\gr B/\gr^\#J\cong \gr C$. Clearly, $\gr^\#J$ is
 an ideal in $\gr B$. Since $\tau$ is
 maximal in $\Gamma$, $P_B(\tau)\cong\Delta(\tau)$ (i.~e., $\Delta(\tau)$ is the projective
 cover of $L(\tau)$ in $\Bmod$. It follows that $\gr\Delta(\tau)$ is the projective cover of
 $L(\tau)$ in $\gr B$-mod. 
 
Since $B$ is a QHA, $J\cong\Delta(\tau)^{\oplus n}$ for some $n>0$.  Next, the head of $\gr^\#J$ has only copies of $L(\tau)$ in it. (According of (\ref{identity}), $\gr C$ is the largest quotient of $\gr A$ with all composition
factors $L(\mu)$, $\mu\in\Theta$. So it is the largest such quotient of $\gr B$.)
Also, there are at most $n$ copies of $L(\tau)$ in this head.
Since $\dim J=\dim \gr^\#J$, it follows that $\gr^\#J\cong\gr\Delta(\tau)^{\oplus n}$. Thus,
$\gr^\#J$ is a heredity ideal and $\gr B$ is a QHA with the indicated standard modules. \end{proof}

\section{Koszul structures on $\gr B$}
 
In this section, let $A$ be a QHA with weight poset $\Lambda$ and let $\Gamma$
be a (non-empty) poset ideal in $\Lambda$. In the notation of \S4, write $\Amod[\Gamma]\cong B$-mod
where $B=A/J$ for a suitable (idempotent) ideal of $A$. Since $B$ is a QHA, it has finite global
dimension. We make the following assumptions.

\begin{hyp}\label{section7hypothesis} Let $A,B$ be as above. The following statements hold.

\begin{itemize}
\item[(1)] There exists a subalgebra $\fa$ of $A$ such that the pair $(A,\fa)$ satisfies Hypothesis \ref{hypothesisofsection6}.
 
\item[(2)] Let $N$ be the global dimension of $B$.  For each $\nu\in\Gamma$, the projective cover $P(\nu)$ of $L(\nu)$ in $\Amod$ is
projective as an $\fa$-module. Moreover, if $N>1$, there is an
exact  complex 
$$0\to\Omega_{N-1}\to P_{N-2}\to \cdots\to P_0\to L(\nu)\to 0$$
in $\Amod$ in which each $P_i$ is a projective $A$-module such that ${P_i}|_{\fa}$ is also projective.
\end{itemize}
\end{hyp}

Let $\fbb$ be the image of $\fa$ under the quotient map $A\to B$.  

\begin{prop}\label{sec6propafterhyp} Assume Hypothesis \ref{section7hypothesis}. For $M,L\in\Bmod$ with $L$ 
completely reducible, we have that the natural maps
$$\begin{cases} (1) \quad \Ext^n_A(M,L)\to\Ext^n_\fa(M,L)\\
(2)\quad \Ext^n_B(M,L)\to\Ext^n_\fbb(M,L)\end{cases}$$
are injective, for all $n\geq 0$.\end{prop}

\begin{proof}  First, we prove the assertion (*) below for $A$-modules $M,L$ with $L$ completely
reducible. 
\begin{itemize}
\item[(*)] Suppose there is a non-negative integer $m$ and an exact complex
$$0\to \Omega_m\to P_{m-1}\to\cdots\to P_0\to M\to 0$$
in $\Amod$   such that each $P_i$ is projective and such that each $P_i|_{\fa}$ is also projective. (We take
this as a vacuous statement if $m=0$.)
Then 
$\Ext^n_{A}(M,L)\to\Ext^n_{\fa}(M,L)$
is injective for $n\leq m+1$.  
\end{itemize}
Assertion (*) follows from Proposition \ref{elemresult} when $m=0$.   So,
without loss, assume that $m>1$ and $n>1$. The kernel $M^{\prime}$ of the 
surjection $P_i\twoheadrightarrow M$ satisfies the hypothesis of (*), provided we 
replace $m$ by $m-1$. However, we may assume by induction that
$\Ext^{n-1}_{A}(M^{\prime},L)\to\Ext^{n-1}_{\fa}(M^{\prime},L)$ is injective for $n\leq m$
(equivalently, $n-1\leq m-1$). But $\Ext^{n-1}_{A}(M^{\prime},L)\cong
\Ext^n_{A}(M,L)$ and $\Ext^{n-1}_{\fa}(M^{\prime},L')\cong\Ext^{n-1}_{\fa}(M,L).$  
This proves the assertion (*).  Its hypothesis is satisfied for $m\leq N-1$ by Hypothesis \ref{section7hypothesis}, while $\Ext^n_A(M,L)=\Ext^n_B(M,L)=0$ for $n>N$ by (\ref{derivedembedding}).  This (1) holds.

 To prove assertion (2) (for $B$-modules), it is only necessary to apply assertion (*), taking $M,L\in\Bmod$, to the top
 horizontal map of the commutative diagram
 $$\begin{CD}
 \Ext^n_{A}(M,L) @>>> \Ext^n_{\fa}(M,L)\\
 @AA{\sim}A  @AAA\\
 \Ext^n_B(M,L) @>>> \Ext^n_\fbb(M,L)\end{CD}$$
 The left-hand vertical map is an isomorphism by (\ref{derivedembedding}). The required injectivity of (2)
 follows.
\end{proof}

In earlier sections, we worked with the standard modules $\Delta(\lambda)$, $\lambda\in\Lambda$,
in the highest weight category $\Amod$.
These results  have evident dual formulations for the modules $\nabla(\lambda)$. These dual
formulations often do not need to be proved separately in   case the category $\Amod$ has a duality ${\mathfrak d}:\Amod\to(\Amod)^{\text{op}}$.  (Thus, $\mathfrak d$ is a contravariant category
equivalence such that ${\mathfrak d}L\cong L$ for all irreducible $A$-modules.) In this case,
${\mathfrak d}\Delta^i(\lambda)\cong\nabla_{-i}(\lambda),$ $ \forall\lambda\in\Lambda, \forall i\in\mathbb N,$
 and   $\Ext_A^\bullet(M,N)
\cong\Ext^\bullet_A({\mathfrak d}N,{\mathfrak d}M)$, $\forall M,N\in\Amod.$ Thus, in the presence of a duality, the verification of condition
(SKL$'$), as stated in \S4, simplifies to just checking just one of the two conditions, say condition (1) involving the $\Delta^i(\lambda)$. In the applications in \S8, a natural duality is present.

\begin{thm}\label{Koszultheorem}
Assume Hypothesis \ref{section7hypothesis} above holds and that $\Amod$ or $\Bmod$ has a duality $\frak d$. Also, assume that $\fa$ is Koszul, and that $\Bmod$ has the (KL) property with respect to a given length
function $l$. Then $B$ satisfies the (SKL$'$) property with respect to $l$.
\end{thm}
\begin{proof} Using (\ref{derivedembedding}), it suffices to show that the condition (SKL$'$) holds for the algebra $A$ as long as $\lambda,\mu\in\Gamma$.  Also,
given $\lambda\in\Gamma$, $L(\lambda)$ has a projective cover in $\Amod$ which is $\fa$-projective.
Thus, by Theorem \ref{MainInfinitesimalThm}, $\Delta(\lambda)$ has a $\fa$-graded structure generated by its grade 0 term and satisfying the property that $\Delta(\lambda)|_{\fa}$ is $A_0$-stable.

  In the discussion below, we take
$\lambda,\mu\in\Gamma$.
It suffices to show  that (SKL$'$)(1) holds for each $i>0$ since the assumed (KL) property (for weights in $\Gamma$) implies the case $i=0$. Denote $\Delta(\lambda)$ by $\Delta$, and let
$\Delta'=\Delta|_\fa\in\fagrmod_0$.   Put $L=L(\mu)$. We will make use of the
following commutative diagram (with the evident maps):

{\small
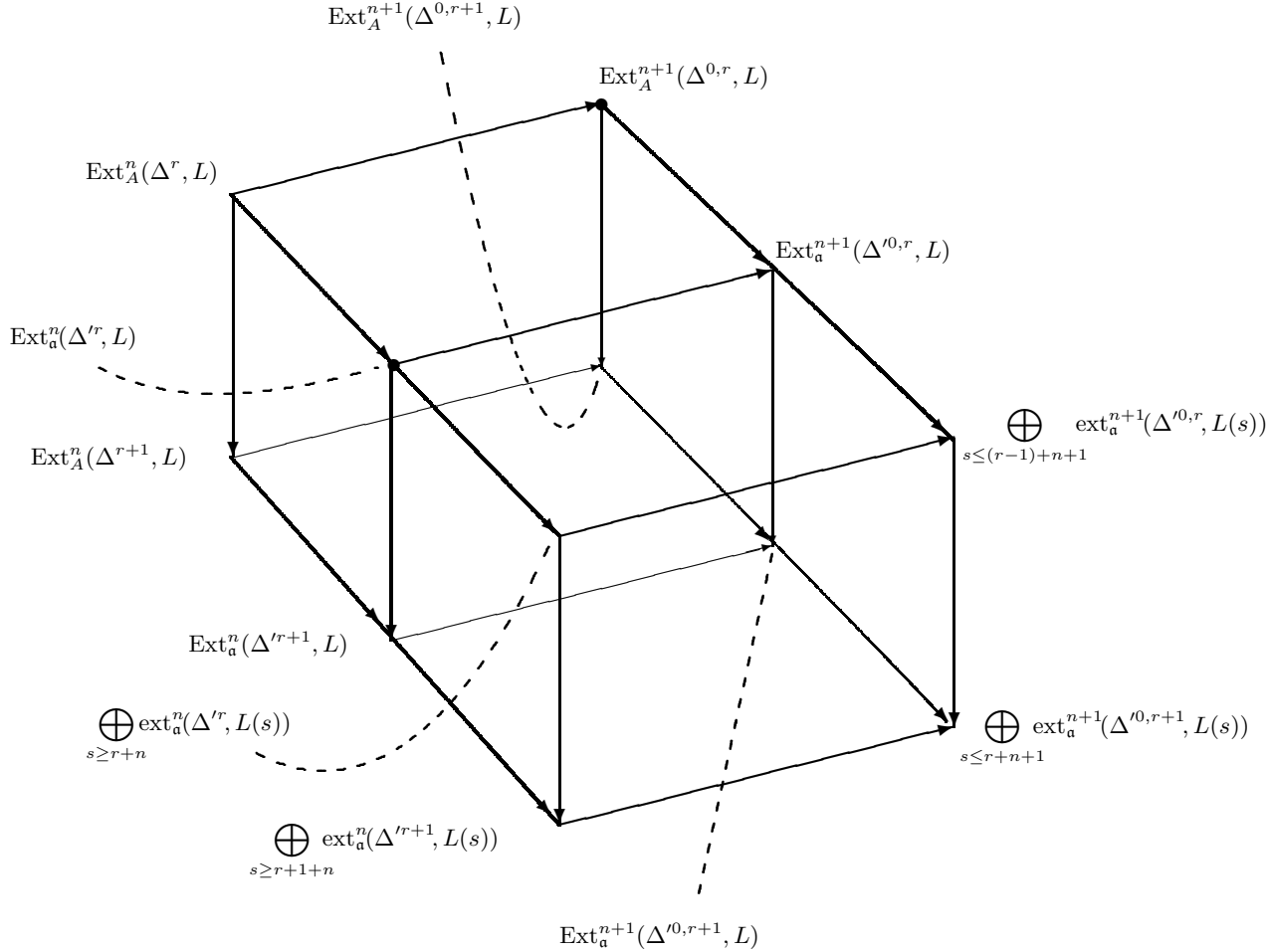
\begin{picture}(450,400)


\thicklines

\put(120,50){\vector(4,1){150}}

\put(120,160){\vector(4,1){150}}

\put(120,160){\vector(0,-1){110}}

\put(270,197){\vector(0,-1){110}}


\thinlines

\put(56,120){\vector(4,1){145}}

\put(201,261){\vector(0,-1){105}}

\thicklines

\put(56,225){\vector(4,1){145}}

\put(56,225){\vector(0,-1){105}}


\thinlines

\put(-4,190){\vector(4,1){140}}

\put(136,324){\vector(0,-1){100}}

\thicklines

\put(-4,290){\vector(4,1){140}}

\put(-4,290){\vector(0,-1){100}}

\bezier{1000}(-5,290)(120,160)(120,160)

\bezier{1000}(-5,190)(120,50)(120,50)

\bezier{1000}(135,325)(270,196)(270,196)

\thinlines

\bezier{1000}(135,225)(267,89)(267,89)

\thicklines

\put(263,94){\vector(1,-1){5}}

\put(195,163){\vector(1,-1){5}}

\put(112,59){\vector(1,-1){5}}

\put(47,132){\vector(1,-1){5}}

\put(114,166){\vector(1,-1){5}}

\put(51,232){\vector(1,-1){5}}

\put(262,204){\vector(1,-1){5}}

\put(194,268){\vector(1,-1){5}}

\put(54,222){\Large $\bullet$}

\put(133,321){\Large $\bullet$}

\put(-60,296){$\Extone$}

\put(135,332){$\Exttwo$}

\put(-89,233){$\Extthree$}

\put(202,266){$\Extfour$}

\put(-60,86){$\Extfive$}

\put(273,200){$\Extsix$}

\put(-81,187){$\Extseven$}

\put(-21,116){$\Extnine$}

\put(2,42){$\Exteleven$}

\put(272,85){$\Exttwelve$}

\put(32,355){$\Exteight$}

\put(120,5){$\Extten$}

\curvedashes[1mm]{0,1,2}

\put(-60,20){\curve(176,136, 120,60, 60,60)}

\put(-30,44){\curve(-25,180, 20,170, 81,180)}

\put(100,44){\curve(-25,300, 10,170, 36,180)}

\put(23,44){\curve(150,-20, 164,45, 178,110)}
\curvedashes{}

\end{picture}
}

To prove (SKL$'$) for $r=i>0$, it suffices to show
\begin{equation}\label{toshow}\Ext^n_A(\Delta^{r,r+1},L)
\to\Ext^n_A(\Delta^r,L)\,\,{\text{\rm is surjective, for $n\geq 0$.}}\end{equation}
 Indeed, assuming (\ref{toshow}), if $\Ext^n_A(\Delta^r,L)\not=0$, there is
an irreducible constituent $L(\tau)$ of $\Delta^{r,r+1}=\Delta^r/\Delta^{r+1}$ such that
$\Ext^n_A(L(\tau),L)\not=0$, whence $n\equiv l(\tau)-l(\mu)$ mod$\,2$. On the other hand, any
irreducible constituent $L(\tau)$ of $\Delta^{r,r+1}$ necessarily satisfies $l(\tau)\equiv r +l(\lambda)$
mod$\,2$,  using  \cite[Cor. 5.4(b)]{CPS1}. So, finally, $n\equiv l(\lambda)=l(\mu)+r$ mod$\,2$, as required in
(SKL$'$). But (\ref{toshow}) is
clear for $n=0$, so assume that $n>0$. Then, relabeling $n$ as $n+1$, it suffices to show
\begin{equation}\label{reallytoshow}\Ext^n_A(\Delta^r,L)\to\Ext^n_A(\Delta^{r+1},L)\quad
{\text{\rm is the zero map,}}\quad \forall r>0, \forall n\in{\mathbb N}.\end{equation}
The maps
$ \Ext^n_A(\Delta^{r+1},L)\to\Ext^{n+1}_A(\Delta^{0,r+1},L)$  and $
\Ext_A^{n+1}(\Delta^{0,r+1},L)\to\Ext^{n+1}_\fa(\Delta^{\prime 0,r+1},L)$ are injective by (\ref{afterSKL})
and Proposition \ref{sec6propafterhyp}, respectively.

Since $\fa$ is Koszul, we can assume $s\leq r+n+1$ in the isomorphism
$\Ext^{n+1}_\fa(\Delta^{\prime 0,r+1},L)\cong\bigoplus_{s\in\mathbb Z}\grExt^{n+1}_\fa(\Delta^{\prime 0,r+1},L(s))$.
 The Koszul property gives similar (indicated) constraints
in the commutative diagram above.
Following  $x\in\Ext^n_A(\Delta^r,L)$ through the maps above, and considering alternative paths in the lower plane, the image  $y$ of
$x$ in
$\bigoplus_{s\leq r+n+1}\grExt_\fa^{n+1}(\Delta^{\prime 0,r+1},L(s))$ lies in  $\grExt^{n+1}_\fa(\Delta^{\prime 0,r+1},L(s))$ with
$s=r+n+1$.
Also, it comes by  a path in the upper plane of the diagram from an element in
$\bigoplus_{s\leq (r-1)+n+1}\grExt^{n+1}_\fa(\Delta^{\prime 0,r+1},L(s))$ with $s=r+n$.
The vertical maps on $\grExt$'s preserve the terms with coefficients $L(s)$ for any given $s$. So  $y=0$ and
the image of $x$ in $\Ext^n_A(\Delta^{r+1},L)$ is zero. Thus, (\ref{reallytoshow}), and hence (\ref{toshow}), hold.
\end{proof}

\begin{cor}\label{grcorollary} Assume the hypothesis of Theorem \ref{Koszultheorem} holds. Then $\gr B$ is a graded QHA with the same poset $\Gamma$ as
$B$, and with standard modules $\gr\Delta(\lambda)$, $\lambda\in\Gamma$. Also, $\gr B$ satisfies the graded Kazhdan-Lusztig
property with respect to $l|_\Gamma$. In particular, $\gr B$ is a Koszul algebra. It has the same
Kazhdan-Lusztig polynomials as $B$.  The standard modules for $\gr B$ are linear (and the costandard modules
satisfy a dual property).  \end{cor}

This follows by combining Theorem \ref{Koszultheorem} and Theorem \ref{SKLtheorem}. The linearity
of standard modules (and the dual property for costandard modules) is essentially part of the definition of a graded Kazhdan-Lusztig property; see \S4 and the discussion above (2.0.4).  

 Suppose that $C$ is a finite dimensional algebra over a field $K$ with representatives $L_1,\cdots, L_m$
 from the isomorphism classes of distinct irreducible modules.  Put $M=\bigoplus L_i$. The homological dual $C^!$ 
 is Morita equivalent to the Yoneda Ext-algebra $\Ext^\bullet_C(M,M)$. 

\begin{cor}\label{koszuldual} Assume the hypotheses of Theorem \ref{Koszultheorem} hold. Then $B^!$-grmod is a graded HWC with a graded Kazhdan-Lusztig theory. It has weight poset $\Gamma^{\text{\rm op}}$ (opposite poset to $\Gamma$),
length function $l^{\text{\rm op}}=-l$, and Kazhdan-Lusztig polynomials
$F_{\nu,\lambda}:=(-t)^{l(\nu)-l(\lambda)}\sum_n[\gr\Delta(\nu):L(\lambda)(n)](-t)^{-n}.$
In particular, $B^!$ is a Koszul algebra. Also, $(\gr B)^!\cong B^!$.\end{cor}

\begin{proof} By Theorem \ref{Koszultheorem},  $B$ satisfies the condition (SKL$^\prime$), so that all the
assertions, except the last follow from \cite[Thm. 3.8]{CPS2}. Finally, by
 \cite[Thm. 2.2.1(a)]{CPS5},  $B^{!!}\cong \gr B$.  Hence, $(\gr B)^!\cong B^{!!!}
\cong B^!$. \end{proof}

The polynomials $F_{\nu,\lambda}$, $\lambda,\mu\in\Gamma$, in Corollary \ref{koszuldual} identify
with the inverse Kazhdan-Lusztig polynomials $Q_{\lambda,\nu}$ associated to the Kazhdan-Lusztig polynomials $P_{\lambda,\nu}$ of $\gr B$-mod
(or equivalently $\Bmod$). These polynomials are defined by the condition that
$\sum_{\lambda\leq \nu\leq\lambda'}(-1)^{l(\lambda')-\l(\nu)}Q_{\lambda,\nu}P_{\nu,\lambda'}=\delta_{\lambda,\lambda'}.$
See \cite[\S3]{CPS1}---where our $Q_{\lambda,\nu}$ here is  denoted $Q_{\nu,\lambda}$.

  In the formula for $F_{\nu,\lambda}$
in Corollary \ref{koszuldual},  $[\gr\Delta(\nu):L(\lambda)(n)]\not=0$ if and only if $L(\lambda)$ appears as a constituent
of $\rad^n\Delta(\lambda)/\rad^{n+1}\Delta(\lambda)$. Thus, in that case,  $l(\lambda)-l(\nu)\equiv n$ mod $2$ (a fact
already used in the proof of Theorem \ref{Koszultheorem}). Therefore, in the
expression for $F_{\nu,\lambda}$, the ``$-t$" can be replaced by ``$t$". We have maintained the $-t$ in the definition
of $F_{\nu,\lambda}$ to be consistent with \cite{CPS1}.\footnote{We take this opportunity to record that the discussion of these polynomials in \cite[(2.1.3)]{CPS5} contains a misprint,
which is corrected by changing the two separate occurrences of $t$ and $-t$ to either both $-t$, or both $t$.}

\begin{rem}  
 If Hypothesis \ref{section7hypothesis} holds with  $B=A$, it is possible to prove a
``bare-bones" version of Koszulity in an elementary way. The argument for
Koszulity of $\gr B$ in this case  does not is require the Hopf algebra setting of \S5
or results of \S6, and is given below in a concise  form.
Nevertheless, while it provides an interesting model, there are reasons
which prevent this simple and elegant result from being a cornerstone in our
theory. First, the hypotheses do not actually occur in any of the regular weight cases of
\S\S8,9,10; a main reason is that we must use appropriately truncated projective modules for $B$. Second, the conclusion does not give us everything we want. Even if the hypotheses were
to apply to conclude that $\gr B$ is Koszul, and if we also assume  the (much deeper) quasi-hereditary
property of $\gr B$ from \S6, we would still not be able to conclude that $\gr
B$-mod satisfied the (KL) property,   even when $B$ has such a theory. Also, we would not have the
other conclusions of Corollaries \ref{grcorollary} and \ref{koszuldual}. All these conclusions are very
valuable for the larger program of this paper, which is to transport all or
most good properties (especially those arising in Lie theory) known for an
algebra $B$ to $\gr B$, while at the same time providing conditions which show
$\gr B$ is Koszul.

    Nevertheless, an alternative approach to Theorem \ref{Koszultheorem} would be welcomed
for establishing these other conclusions, and we raise the existence of
effective alternatives as an open question.  We hope to provide some
variations ourselves in a future work  \cite{PS5}
   studying gradings in the setting of integral quasi-hereditary algebras
and highest weight categories \cite{CPS2} and \cite{DS}. \end{rem}

\begin{prop}\label{KosFact} Suppose $B$ is a finite dimensional algebra with a
subalgebra $\fa$ satisfying the conditions.

\begin{itemize}
\item[(1)]
$\fa$ is positively graded and Koszul,

  \item[(2)] $(\rad \fa)B=\rad B$.
  
\item[(3)]   Every projective $B$-module is projective as an $\fa$-module. (Equivalently,
the $\fa$-module $B$ is projective as a left $\fa$-module.)\end{itemize} 
   Then $\gr B$ is Koszul.\end{prop}

 \begin{proof} If $P$ is a projective $B$-module,  $\gr P$ is projective as a $\gr B$
module and as an $\fa$-module (as follows from the projectivity of $\gr \fa\cong
\fa$), and every projective $\gr B$ module arises this way. If $P$ has
a simple head as a $B$-module, then $\gr P$ has a simple head as a $\gr B$ -module,
and is semisimple as an $\fa$-module; also $\gr P$ may be regarded as a graded projective module
generated in degree $0$ (for either $\gr B$ or $\fa$).
   Henceforth, we will just consider projective modules for $\gr B$, and just
denote them by $P$ to simplify notation. For any irreducible $\gr B$-module $L$, regarded as graded with pure degree $0$, let
                      $\cdots\to P_n\to \cdots \to P_0\to L\to 0$
   be a minimal graded projective resolution of $L$. This resolution is also
minimal as an ungraded projective resolution, and as a graded or ungraded
projective resolution
   of the (completely reducible) restriction of $L$ to $\fa$.  Thus, by
Koszulity of $\fa$, the head of  $P_n$ is pure of degree $n$ as an $\fa$-module.
This graded module $P_n$, and its (graded) head, are just the restriction to $\fa$ of
the graded module $P_n$ for $\gr B$ and its $\gr B$ head. Thus, the
(graded) head of $P_n$, viewed as a graded  $\gr B$-module, has grade $n$.  Since this
holds for all irreducible modules $L$ for $\gr B$, the algebra $\gr B$ is Koszul. \end{proof}

\vskip.1in
\begin{center}{{\large\bf Part II: Applications}}\end{center}
 
Part II, which consists of \S\S8--10, takes up applications of Part I to specific finite dimensional algebras
 arising in representation theory. These algebras include:

\begin{itemize}
\item[(1)]  Quasi-hereditary algebras $B=U_\Gamma$
attached to a quantum enveloping algebras $U=U_\zeta$ in characteristic 0 at an $e$th root of unity ($e$ odd). For the definition of the quotient algebra $U_\Gamma$ of $U$, see (\ref{Asubgamma}). Here the non-empty, finite weight poset $\Gamma$ is required to consist of  $e$-regular weights. Thus,  it must hold that $e\geq h$ (the Coxeter number
of $U_\zeta$; in fact, it will be necessary to assume that $e>h$).
See \S8.

\item[(2)] Regular blocks of $q$-Schur algebras $S_q(n,r)$. This section is largely a formal recasting of
the results in (1) in the special case of type $A$. See \S9.

\item[(3)] Quasi-hereditary algebras $B=U_\Gamma$ attached to the distribution algebra $U={\text{\rm Dist}}(G)$ of a simple, simply connected algebraic group $G$ in (large) positive
characteristic $p$; see \S10 for more details, including the definition of $U_\Gamma$. Here also the non-empty finite weight poset $\Gamma$ is required to consist of $p$-regular weights. 

\end{itemize}

\medskip 
In these cases, considerable effort is required to establish the various  hypotheses needed
for the results of Part I.  For example, in case (3), the validity of the Lusztig character formula (LCF) for regular weights
in the Jantzen region 
is essential.  (See Remark \ref{LCF}.)  
With current technology, this validity is known for $p$ very large. See \cite{Fiebig} for a specific
bound.  It is expected that $p\geq h$ will ultimately be shown to suffice. The validity of the LCF means that, by 
 Andersen-Jantzen-Soergel \cite{AJS}, with $p>h$, the $p$-regular blocks of the restricted universal enveloping algebra $u$ are Koszul;  also for $e>h$, the regular blocks of the small quantum group $u_\zeta$ of $U=U_\zeta$  are also Koszul \cite{AJS}.

The conditions given in Hypotheses \ref{hypothesisofsection6} and  \ref{section7hypothesis} need to be checked. In this setup, there is a pair $(A,\fa)$ as in Part I, and $B=U_\Gamma$, where $\Gamma$ is a poset ideal in
the weight poset $\Lambda$ of $A=U_\Lambda$. (Thus, $B=A/J$ for a defining ideal $J$ of $A$, and $B$ also equals $A_\Gamma$ in the notation of \S4.)  In fact, the process begins with $B$ and $\Gamma$, and then $A$ and
$\Lambda$ are chosen to guarantee the hypotheses. The verification takes some effort, especially in the
positive characteristic case of \S10. Meaningful results here require at least that $p\geq 2h-2$ and that $p$
is large enough so that the LCF holds in the Jantzen region. Additional requirements on $p$ also appear in
\S10, such as $p\geq 4h-5$ in Theorems \ref{alternativetheorem} and \ref{filteringWeyl}, and even
$p\geq 2N(h-1)-1$, for a certain $N$ with $2\leq N\leq(h-1)|\Phi|$  in Theorem \ref{charpversion}. However, it is expected that all the additional requirements on $p$ can be
eventually removed (leaving only $p\geq 2h-2$ and  $p$ large enough so that the LCF holds in the Jantzen
region).

In each case, once the above program is carried out, the results of \S7 are available for $\gr B$, as 
discussed in the introduction.  

  \section{Quasi-hereditary algebras in the quantum enveloping algebra case}
 Let $\Phi$ be a finite root system, spanning a Euclidean space $\mathbb E$ having
 inner product $(x,y)$. For convenience, assume $\Phi$ is irreducible. For $\alpha\in\Phi$, let
 $\alpha^\vee:=\frac{2}{(\alpha,\alpha)}$.  Let $\mathfrak g$ be the complex simple
  Lie algebra with root system $\Phi$. Choose a set $\Pi=\{\alpha_1,\cdots, \alpha_s\}$ of simple roots, and
  let  $\Phi^+$ the corresponding set of positive roots. For $1\leq i\leq s$,  $\varpi_i\in\mathbb E$ denotes the
 fundamental dominant weight defined by  $(\varpi_i,\alpha_j^\vee)=\delta_{i,j}$, $1\leq j\leq s$; let $X^+={\mathbb N}\varpi_1\oplus\cdots\oplus{\mathbb N}\varpi_s$ be the cone of dominant weights in $X:={\mathbb Z}\varpi_1\oplus\cdots
\oplus {\mathbb Z}\varpi_s$.    Let $\rho=\varpi_1+\cdots+\varpi_s$ and
 $h=(\rho,\alpha_0^\vee)+1$ (Coxeter number), where $\alpha_0\in\Phi^+$ is the maximal short root.
 
 Fix an integer $e>1$. Some restrictions, discussed below, will be made on $e$.
 
  For $\lambda,\mu\in X^+$, define $\lambda\leq \mu\iff
 \mu-\lambda$ is a sum of positive roots; as such, $X^+$ is a poset. If $Y\subseteq X^+$,
 it is regarded as a poset by restricting the relation $\leq$ to it---thus, 
 $Y$ is a subposet of $X^+$. For example,  $\Gamma_{e{\text{\rm -reg}}}=\Reg$ denotes the
 subposet of $X^+$ consisting of $e$-regular dominant weights, i.~e., $\lambda\in\Reg$ if and only if $\lambda\in X^+$ and
 $(\lambda+\rho,\alpha^\vee)\not\equiv 0$ mod$\,e$,  $\forall\alpha\in\Phi$.  Then $\Reg\not=\emptyset\iff
 e\geq h$.    Put $X^+_1=X^+_{1,e}:=\{\lambda\in X^+\,|\,0\leq (\lambda,\alpha^\vee)<e,\,\forall\alpha\in\Pi\,\}$,
 the set  of $e$-restricted dominant weights.  
  
 If $Z\not=\emptyset$ is a subset of a poset $Y$, then $\{y\in Y\,|\,y\leq z,\,\,\text{\rm for some $z\in Z$}\}$ is called the
 poset ideal in $Y$ generated by $Z$. Then $\Gamma_{\text{\rm $e$-res}}=\Res$ is the poset ideal in
 $X^+$ generated by $X^+_1$. 
  Also, 
 let $\Resreg=\Reg\cap \Res$, the poset ideal in $\Reg$ generated by the $e$-regular, $e$-restricted weights. (Caution: $\Res$ does not generally consist entirely of $e$-restricted weights.)

 {\it For the rest of this section, unless otherwise noted, the following assumptions on the integer $e$ will be
 required.}  First, $e>1$ is odd.\footnote{It seems likely that this assumption can be removed. This assumption is only made so that we can apply \cite{AJS}, written before treatments of quantum groups allowing
 even roots of 1. The case when $e=2e'$ with $e'$ odd  can be easily treated, using \cite{A} and the 
 Comparison Theorem 
 \cite{PS-2}, \cite[Thm. 2]{CPS8}.}  We require also that $e> h$.   In addition, in types $B,C$ and $F_4$, assume that $e$ is not divisible by $4$, and in type $G_2$, assume
 that $e$ is not divisible by 3. 
 
    Let $K={\mathbb Q}(\zeta)$ be the cyclotomic field obtained by adjoining a primitive $e$th root of unity
 $\zeta=\sqrt[e]{1}$ to $\mathbb Q$.
 Let ${\mathbb U}={\mathbb U}({\mathfrak g})$ be the  quantum enveloping algebra over the function field ${\mathbb Q}(v)$ associated to $C$.\footnote{For further details on the quantum enveloping algebra, especially at roots of unity used below, including the small quantum group, see \cite{Jan2}, \cite{Lus1}, \cite{Lus2}, \cite{APW1}, \cite{APW2}.} It has generators $E_1,\cdots, E_s, F_1,\cdots, F_s,
 K_1^{\pm 1}, \cdots, K_s^{\pm 1}$ satisfying the familiar relations. Let $U_{\mathbb Z}$ be the ${\mathbb Z}[v,v^{-1}]$-form of $\mathbb U$ generated by
 the ``quantum" divided powers $E_i^{(m)}, F_i^{(m)}$, together with the elements $K_i^{\pm 1}$, $1\leq i\leq s$, $m\in\mathbb N$. Let $U'=K\otimes_{{\mathbb Z}[v,v^{-1}]}U_{\mathbb Z}$, with $K$  viewed as a ${\mathbb Z}[v,v^{-1}]$-algebra by specializing $v\mapsto \zeta$. Finally, let
 $U_\zeta=U'/\langle K_1^e-1,\cdots, K_s^e-1\rangle$ be the quotient of $U'$ by the ideal generated by the central elements $K_i^e-1$. 
 
 There is a (surjective) Frobenius morphism $\Fr:U_\zeta\to U({\mathfrak g})$ from $U_\zeta$ to the
 universal enveloping algebra $U({\mathfrak g})$ of $\mathfrak g$,  whose Hopf-theoretic kernel $u_\zeta$ is the small quantum group (of dimension $e^{\dim({\mathfrak g}})$.  Also, $u_\zeta$ is a normal (Hopf) subalgebra of $U_\zeta$ such that $U_\zeta//u_\zeta\cong U({\mathfrak g})$. If $M$ is $U({\mathfrak g})$-module, let $M^{(1)}$ be the pull-back
 of $M$ through Fr.  

The category $U_\zeta$-mod of finite dimensional, integrable, type 1 $U_\zeta$-modules is a HWC with poset $X^+$ having
 standard (resp., costandard, irreducible) modules denoted $\Delta_\zeta(\lambda)$ (resp., $\nabla_\zeta(\lambda)$, $L_\zeta(\lambda)$), $\lambda\in X^+$.  Our restrictions on $e$ mean that
   for $\lambda\in X^+$, the character $\ch L_\zeta(\lambda)$ is given by the LCF. See \cite[\S7]{T}.\footnote{The length function is explained above Theorem \ref{KoszulforgeneralizedqSchur}.}   For $\lambda\in X^+$,
   then $L(\lambda)^{(1)}\cong L_\zeta(e\lambda)$.

  The set $X^+_1$  of $e$-restricted dominant weights indexes the irreducible $u_\zeta$-modules; in fact, given $\lambda\in X^+_1$, $L_\zeta(\lambda)|_{u_\zeta}$ is the corresponding irreducible $u_\zeta$-module.
  Given $\lambda\in X^+_1$, we let $Q_\zeta(\lambda)\in U_\zeta$-mod be the projective cover
 for $L_\zeta(\lambda)$. As a $u_\zeta$-module, $Q_\zeta(\lambda)$ is the projective cover for $L_\zeta(\lambda)|_{u_\zeta}$ \cite[\S4.6 Theorem]{APW2}. Furthermore,
 $Q_\zeta(\lambda)\cong T_\zeta(2(e-1)\rho+w_0\lambda)$
  is
the indecomposable tilting module in $U_\zeta$-mod of highest weight $2(e-1)\rho+w_0\lambda$, where $w_0\in W$ has maximal length in
 the Weyl group $W$ of $\Phi$.

 The following standard result follows easily from the tensor product theorem. (For
 the $\Leftarrow$ direction, argue on $u_\zeta$-isotypic components as in \cite[I, (2.14)(3)]{JanB}.)

 \begin{lem}\label{completelyreducible} Any $M\in U_\zeta$-mod is completely reducible if and only if 
 its restriction $M|_{u_\zeta}$ to $u_\zeta$ is completely reducible.\end{lem}

  Suppose $\Gamma$ is a finite non-empty poset ideal in $\Reg$, and let $J_\Gamma$ be the annihilator in $U_\zeta$ of the full subcategory $U_\zeta{\text{\rm -mod}}[\Gamma]$ of $U_\zeta$-mod. 
 Put 
 \begin{equation}\label{Asubgamma} A:=U_\Gamma:=U_\zeta/J_\Gamma\end{equation}
  and let $\fa$ be the image of $u_\zeta$ in $A$. The algebra $U_\zeta$ has an anti-involution $\iota$ (stabilizing $u_\zeta$) such that, for $\lambda\in X^+$,
 $L_\zeta(\lambda)^{*\iota}\cong L_\zeta(\lambda^\star)$, where $\lambda^\star:=-w_0\lambda$.  Composing
 $\iota$ with an appropriate graph automorphism yields an anti-involution $\iota'$ which induces   a duality $\mathfrak d$ on
 the category $\Amod$.  There is an induced duality  on $\gr A$-mod 
 and for $\gr A$-grmod. (In the latter case, a module of pure grade $n$ has a dual in pure grade $-n$.)

 The standard and costandard modules for a QHA algebra $A$ are determined by its weight poset
 of $A$.\footnote{We have stated this result for a finite ideal of $e$-regular weights, but, in fact, it remains valid
 for any finite ideal in the larger poset $X^+$. Since most results are established for regular weights, it seems simplest to concentrate on this case.}  For details on the following result, see \cite[Prop. 3.5, Cor. 2.5, Prop. 2.1.2]{DS}.

 \begin{lem}\label{DSlemma}  For any finite non-empty poset ideal $\Gamma$ in $\Reg$, $A=U_\Gamma$ is a QHA (hence, finite dimensional) satisfying $\Amod
 \cong U_\zeta{\text{\rm -mod}}[\Gamma]$. The standard, costandard, irreducible modules in the HWC 
 $\Amod$ (with poset
 $\Gamma$) are the $\Delta_\zeta(\lambda)$, $\nabla_\zeta(\lambda)$, $L_\zeta(\lambda)$, $\lambda\in\Gamma$.\end{lem}

 A block $E$ in the finite dimensional algebra $u_\zeta$ is called $e$-regular if its irreducible modules
 are of the form $L_\zeta(\lambda)|_{u_\zeta}$ for $\lambda$ an $e$-restricted and $e$-regular dominant
 weight. 
  Let
 \begin{equation}\label{regularsubalgebra}
 u'_\zeta =\prod_{E\,e-{\text{\rm regular}}} E,\quad\text{\rm
a direct factor of $u_\zeta$}.
\end{equation}

 A finite non-empty poset ideal $\Gamma$ in $\Reg$ is called {\it $e$-fat} if, given any $e$-regular, $e$-restricted 
 dominant weight $\lambda$,  $2(e-1)\rho+ w_0(\lambda)\in\Gamma$. The reader can verify that,
 if $\Gamma$ is $e$-fat, then it contains all $e$-dominant, $e$-restricted weights. Moreover,  for
 such a $\lambda$, the $U_\zeta$-projective cover
   $Q_\zeta(\lambda)$ of $L_\zeta(\lambda)$ belongs
  to $U_\zeta{\text{\rm -mod}}[\Gamma]$, the full subcategory of $U_\zeta$-mod with modules
 whose composition factors $L_\zeta(\gamma)$ satisfy $\gamma\in\Gamma$. This is because
 $2(e-1)\rho + w_0(\lambda)$ is the highest weight of the $U_\zeta$-module $Q_\zeta(\lambda)$.

 \begin{lem}\label{condition} Let $\Gamma$ be an $e$-fat poset ideal in $\Reg$  and let $A:=U_\Gamma$. Then under the quotient map $\phi:U_\zeta\twoheadrightarrow A$, the algebra $u'_\zeta$ maps
 isomorphically onto a normal subalgebra $\fa$ of $A$. 
 
 Furthermore, $\fa$ is a Koszul algebra satisfying $(\rad\fa)A
 =\rad A$.\end{lem}

 \begin{proof}Since $\Gamma$ is $e$-fat, for any $\lambda\in\Resreg$,  we have $Q_\zeta(\lambda)\in\Amod$. Thus, $u'_\zeta$ acts faithfully on
 the category $\Amod$, so that $u'_\zeta\cap J_\Gamma=0$. Hence, $\phi$ maps $u'_\zeta$ isomorphically onto its image $\fa$ in $A$, which must necessarily
 be a subalgebra (i.~e., it contains the identity of $A$). Since the LCF holds for $U_\zeta$,
\cite[Prop. 18.17, p. 256]{AJS} shows that the algebra $\fa\cong u'_\zeta$ is Koszul.  In addition,
$\fa$ is the image of the normal subalgebra $u_\zeta$ under the map $U_\zeta\to A$, so that 
$\fa$ is normal in $A$.

It remains to prove that $(\rad \fa)A=\rad A$. For this, it is enough to check that
$(\rad \fa )P=\rad P$, for each projective indecomposable module $P\in\Amod$. However, such a PIM
$P$ can be constructed as the largest $A$-module quotient of a $U_\zeta$-module
$Q_\zeta(\lambda)=Q_\zeta(\lambda_0)\otimes L(\lambda_1)^{(1)}$, where $\lambda=\lambda_0+e\lambda_1
\in\Gamma$, with $\lambda_0\in X^+_1$, $\lambda_1\in X^+$.   But $\rad_{U_\zeta}Q_\zeta(\lambda)=(\rad_{U_\zeta}Q_\zeta(\lambda_1))\otimes L(\lambda_1)^{(1)}=\rad_{u_\zeta}(Q_\zeta(\lambda))$, and the desired equation
follows by passing to the homomorphic image $P$. We use here that $Q_\zeta(\lambda_0)$ is also a PIM
in the category of $u_\zeta$-modules.
 \end{proof}

The equation $(\rad\fa)A=\rad A$ at the end of the lemma does not require that
$\Gamma$ be $e$-fat, since radicals are preserved under homomorphic images. It will be useful, both
in this section and in \S10 (with $e=p$),
to formulate a ``fattening" procedure. If $\xi\in X^+$, write $\xi=\xi_0+e\xi_1$, for $\xi_0\in X^+_1$ and $\xi_1\in X^+$. Put 
\begin{equation}\label{definitionoff} {\mathfrak f}_e(\xi)=2(e-1)\rho+w_0\xi_0+e\xi_1\in X^+.
\end{equation}
Thus, ${\mathfrak f}_e(\xi)$
is the highest weight of the $U_\zeta$-module $Q_\zeta(\xi):=Q_\zeta(\xi_0)\otimes L(\xi_1)^{(1)}$.
For $\emptyset\not=\Psi\subseteq X^+$, let $\Psi(-1)$ be the poset ideal in $X^+$  generated by $\Psi$, and, for $n\geq 0$, define $\Psi(n)$ be the poset ideal in $X^+$ generated by all
${\mathfrak f}_e(\xi)$, with $\xi\in\Psi(n-1)$.\footnote{If $\xi$ is $e$-regular, then so is ${\mathfrak f}_e(\xi)$.
If $\Psi\subseteq \Reg$, then take $\Psi(n)$, $n\geq -1$, to be the poset ideal in $\Reg$ generated by all
${\mathfrak f}_e(\xi)$ for $\xi\in\Psi(n-1)$.}
   In particular, if $\xi\in\Psi$, then the composition factors
of $Q_\zeta(\xi)$ all have highest weights in $\Psi(0)$. More generally, if $M$ is a $U_\zeta$-module
all of whose composition factors $L_\zeta(\xi)$ satisfy $\xi\in\Psi$, then there is, for any non-negative
integer $n$, an exact sequence
\begin{equation}\label{neededresolution}
0\to\Omega_n\to P_{n-1}\to\cdots \to P_0\to M\to 0\end{equation}
of $U_\zeta$-modules, in which all the modules $P_i$ are projective $U_\zeta$-modules and
all 
composition factors of $P_i$ have highest weights in $\Psi(i)\subseteq\Psi(n-1)$. 

Observe that the $P_i|_{u_\zeta}$, $i=0,\cdots, n-1$, are projective modules in $u_\zeta$-mod.

Let $W_e=W\ltimes e{\mathbb Z}\Phi$ be the affine Weyl group generated by
the Weyl group $W$ and the normal subgroup of translations by $e$-multiples of roots. 
Given $\lambda \in X$, there exists a unique $\lambda^-\in \overline{C^-}$ (the closure of $C^-$) such that $\lambda=w\cdot\lambda^-:=w(\lambda^-+\rho)-\rho$ for some
$w\in W_e$.  Among all possible $w\in W_e$ satisfying $w\cdot\lambda^-=\lambda$, let $w_\lambda$ have
minimal length. Thus, if $W_{e,\lambda^-}$ is the stabilizer in $W_e$ of $\lambda^-$, $w_\lambda$ is a minimal
left coset representative of the parabolic subgroup $W_{e,\lambda^-}$ in $W_e$. Also, define
$l:X\to \mathbb Z$ by putting $l(\lambda):=
l(w_\lambda)$. Given dominant weights
$\lambda,\nu$, put $P_{\nu,\lambda}=0$ unless $\lambda^-=\nu^-$. If $\lambda^-=\nu^-$, set $P_{\nu,\lambda}=
P_{w_\nu,w_\lambda}$, the Kazhdan-Lusztig polynomial for the elements $w_\nu,w_\lambda$ in the Coxeter
group $W_e$. Also, 
define $Q_{\nu,\lambda}=Q_{w_\nu,w_\lambda}$ if $\lambda^-=\nu^-$, and $=0$ if $\lambda^-\not=\nu^-$.
Here,  $Q_{x,y}\in{\mathbb Z}[t^2]$ is the inverse Kazhdan-Lusztig polynomial associated to the
pair $(x,y)$.

 The theorem below concerns the category of $U_\zeta$-modules which involve only $e$-regular dominant
 weights.  If
 $\Gamma$ is a finite non-empty poset ideal in $\Reg$, the $P_{\lambda,\nu}\in{\mathbb Z}[t^2]\subseteq{\mathbb Z}[t]$ are the Kazhdan-Lusztig polynomials of the HWC $B=U_\Gamma$.  Because the LCF holds for all $\lambda\in\Gamma$, the HWC $B$-mod satisfied the
 (KL) property with
 respect to $l:\Gamma\to\mathbb Z$, $\lambda\mapsto l(w_\lambda)$; see \cite[Rem. 5.8]{CPS1}. In other words, the $i=0$ part of condition (SKL$^\prime$) given
 in \S4 holds.

\begin{thm}\label{KoszulforgeneralizedqSchur}  Let $\Gamma$ be a finite non-empty poset ideal in $\Reg$ and let $B=U_\Gamma$.
 The category of ungraded modules for the graded algebra $\gr B$ is a HWC with poset $\Gamma$ and a
 duality $\mathfrak d$, which also induces a duality on $\gr B$-grmod. 
 The category $\gr B$-grmod  satisfies the graded Kazhdan-Lusztig property with the same Kazhdan-Lusztig polynomials as $\Bmod$ (which
 are, in turn classical Kazhdan-Lusztig polynomials of $W_e$.)
In particular, $\gr B$ is a QHA with a Koszul grading. For $\lambda\in\Gamma$, the associated standard module is $\gr\Delta_\zeta(\lambda)$, which  is linear as a graded $\gr B$-module.  
\end{thm}

\begin{proof} We will apply Theorem \ref{Koszultheorem}. Let $N$ be the global dimension of $B$.
We can assume that $N>0$, and we set $\Lambda:=\Gamma(N-1)\cup\Resreg(0)$ and $A:=U_\Lambda$. Let $\fa$
be the image of $u'_\zeta$ in $A$. We must
check that Hypothesis \ref{section7hypothesis} holds for the QHAs $A$ and $B$, and the
subalgebra $\fa$ of $A$.

By construction, $\Lambda$ is $e$-fat. Then by Lemma \ref{condition} (with $\Lambda$ here playing the role
of $\Gamma$ there) says that  $u'_\zeta\cong\fa$ is Koszul (thus, tightly graded), and a normal subalgebra of $A$.  (In fact, the map $U_\zeta\twoheadrightarrow A$ maps the Hopf
subalgebra $u_\zeta$ onto $\fa$.)  The lemma also states that $(\rad \fa)A=\rad A$. Thus, $(A,\fa)$
satisfies Hypothesis \ref{hypothesisofsection6}.
 
Therefore, Hypothesis \ref{section7hypothesis}(1) holds. The discussion above for (\ref{neededresolution}) with $n=N-1$ implies that Hypothesis \ref{section7hypothesis}(2) is satisfied. Finally, as observed just
before the statement of the theorem, $\Bmod$ satisfies the (KL) property with with respect to
the length function $l:\Gamma\to\mathbb Z$, $\lambda\mapsto l(\lambda)=l(w_\lambda)$. 

The duality
hypothesis in Theorem \ref{Koszultheorem} is satisfied,  by the discussion above Lemma \ref{DSlemma}. Therefore, 
all the hypotheses of Theorem \ref{Koszultheorem} hold, so it implies that $\Bmod$ satisfies the (SKL$^\prime$) property.
 
  Therefore, the conclusion of the theorem follows from Corollary \ref{grcorollary}. \end{proof}

The theorem implies results on the homological dual $B^!$ and on the composition factors
of the radical sections of the modules $\Delta_\zeta(\lambda)$. These facts are contained in parts (a) and (c) of the
following result. Using translation functors, it is possible to obtain   information in the case of non-$e$-regular
weights---this is the content of part (b). In the result below, the $Q_{\lambda,\nu}$ are the inverse
Kazhdan-Lusztig polynomials discussed right above Theorem \ref{KoszulforgeneralizedqSchur}.

\begin{cor}\label{inverse} (a) Let $B$ be as in the previous theorem and let $B^!=\Ext^\bullet_B(B/\rad B,B/\rad B)$ be its homological dual. Then $B^!$ is a graded QHA having a graded Kazhdan-Lusztig theory
with length function $l^{\text{\rm op}}=-l:\Gamma^{\text{\rm op}}\to\mathbb Z$ and Kazhdan-Lusztig polynomials $F_{\lambda,\nu}:=
Q_{\nu,\lambda}$, for $\nu<\lambda$ in $\Gamma$. In particular, $B^!$ is a Koszul algebra. Also,
$B^!\cong (\gr B)^!$.

(b) Suppose $\lambda\in X^+$ (not necessarily $e$-regular). Then $\Delta_\zeta(\lambda)$ has a $U_\zeta$-filtration
$$\Delta_\zeta(\lambda)=F^0(\lambda)\supseteq F^1(\lambda)\supseteq\cdots\supseteq F^m(\lambda)=0$$
in which each section $F^i(\lambda)/F^{i+1}(\lambda)$, $0\leq i<m$, is a completely reducible $U_\zeta$-module
satisfying the following property: for $n\geq 0$ and $\nu\in X^+$, the multiplicity of $L_\zeta(\nu)$ in the section
$F^n(\lambda)/F^{n+1}(\lambda)$ equals  the coefficient of $t^{l(\lambda)-l(\nu)-n}$ in the polynomial $Q_{\nu,\lambda}$.

(c)
If $\lambda\in 
\Reg$, then in (b) above, we can take $F^n(\lambda)=\rad^n\Delta_\zeta(\lambda)$.

\end{cor}

\begin{proof} Parts (a) and (c) follow from Corollary \ref{koszuldual} and the discussion following it.  Now consider
(b), and suppose that $\lambda^-\in \overline{C^-}\backslash C^-$ is not $e$-regular. Let $T$ be the translation
operator from the category of $U_\zeta$-modules having composition factors with highest weights of the form $w\cdot(-2\rho)\in X^+$ to the category of $U_\zeta$-modules having composition factors with highest weights of the form $w\cdot\lambda^-\in X^+$.
Then $T\Delta_\zeta(w_\lambda\cdot(-2\rho))\cong\Delta_\zeta(\lambda)$. Also, if $\lambda=w\cdot\lambda^-\in X^+$, then $TL(w\cdot(-2\rho))=L(\lambda)$ if and only if $w=w_{\lambda}$, and
equals 0, otherwise.

 Let $F^i(\lambda):=T\rad^i\Delta_\zeta(\lambda)$.
Since $T$ is exact, $F^\bullet(\lambda)$ is a $U_\zeta$-filtration of $\Delta_\zeta(\lambda)$ with completely
reducible sections. In fact, $F^i(\lambda)/F^{i+1}(\lambda)\cong T\rad^i\Delta_\zeta(\lambda)/\rad^{i+1}\Delta_\zeta(\lambda)$. Suppose that $n\geq 0$ and $\nu\in X^+$. To determine the multiplicity of $L_\zeta(\nu)$ in $F^n(\lambda)/F^{n+1}(\lambda)$, we   assume $\nu^-=\lambda^-$ (otherwise, the multiplicity is zero). Then $\nu$ is in the upper $e$-closure
of  $w_{\lambda^-}\cdot C^-$.  In fact, $w_{\lambda^-}\cdot C^-$ is the unique alcove containing $\nu$ in
its upper closure.  Since the multiplicity of $L_\zeta(w_\lambda\cdot(-2\rho))$ in $\rad^n\Delta_\zeta(\lambda)/\rad^{n+1}\Delta_\zeta(\lambda)$ is the coefficient of $t^{l(\lambda)-l(\nu)-b}$ in $Q_{w_\nu,w_\lambda}=Q_{
w_\nu\cdot(-2\rho),w_\lambda\cdot(-2\rho)}$, part (b) is proved.
\end{proof}

\begin{rems} (a) Translating
into the language of quantum enveloping algebras yields a remarkable fact. Let
${\mathbb E}^!:=\bigoplus_{\lambda,\mu\in\Reg}\Ext^\bullet_{U_\zeta}(L_\zeta(\lambda),L_\zeta(\mu))$
be full $\Ext$-algebra of $U_\zeta$ on irreducible modules having regular highest weights. Then ${\mathbb E}^!$
(an algebra without identity) is ``almost Koszul" in the sense that if $e=e_\Gamma$ is the idempotent projection
corresponding to a poset ideal $\Gamma$ in $\Reg$, then $e{\mathbb E}^!e$ is Koszul. 

(b) In \cite{AJS}, it is conjectured that the full small quantum enveloping algebra $u_\zeta$ is Koszul. This conjecture would provide a step toward establishing
Theorem \ref{KoszulforgeneralizedqSchur} (and Corollary \ref{inverse}) for the singular blocks, although our approach would still require  that the singular blocks satisfy a parity condition (have a Kazhdan-Lusztig theory \cite{CPS1}).

(c) In \cite[Conj. 4.1]{Lin0}, Lin conjectured (in the context of algebraic groups in positive characteristic)
an equality of $\Ext^1$-groups between irreducible modules having regular weights and irreducible modules
 having weights on a wall. If this is true in the quantum enveloping algebra setting, it seems likely that translation operators from a ``regular" standard module to a ``singular"
  standard module would preserve radical filtrations. In that case, Corollary \ref{inverse}(c) extends to arbitrary standard modules.
\end{rems}

\begin{thm}\label{gradingresult}Let $\lambda\in \Reg$. There is a positive grading
$\Delta_\zeta(\lambda)=\bigoplus_{n\geq 0}\Delta_\zeta(\lambda)_n$
of $\Delta_\zeta(\lambda)$ as a module for the Koszul algebra $u'_\zeta$. The graded module $\Delta_\zeta(\lambda)$
is generated in degree 0. As a $u'_\zeta/\rad u'_\zeta$-module
$\Delta_\zeta(\lambda)_0\cong L_\zeta(\lambda).$
\end{thm}

\begin{proof} Choose an $e$-fat poset ideal $\Lambda$ in $\Reg$  containing $\lambda$. Then Theorem \ref{MainInfinitesimalThm}  applies.\end{proof}

 Theorem \ref{MainInfinitesimalThm} contains more information. In particular, a Wedderburn complement $A_0$
 for $A=U_\Lambda$  can be chosen, containing $\fa_0$, so that $\Delta_\zeta(\lambda)_0$ is $A_0$-stable.

\section{The Dipper-James $q$-Schur algebras $S_q(n,r)$ in characteristic 0}   Here $n,r$ are positive integers, $q=\zeta^2$ for $\zeta$ a primitive
$e$th root of 1 for an integer $e>n=h$. As in the previous section, we assume $e$ is odd, unless otherwise explicitly noted. Let $\Lambda^+(n,r)$ (resp., $\Lambda(n,r)$) be the set of partitions (resp., compositions) of $r$ with at most $n$ nonzero parts. The dominance order $\unlhd$ on partitions makes $\Lambda^+(n,r)$ into a poset. Let $\Lambda^+_{\text{\rm $e$-reg}}(n,r)$ be the set of $e$-regular partitions in $\Lambda^+(n,r)$,
i.~e., $\lambda\in\Lambda^+(n,r)$ is $e$-regular if and only if no nonzero part of $\lambda$ is repeated $e$ times.\footnote{Also, $\lambda$ is $e$-restricted (or just restricted) if its transpose $\lambda'$ is
$e$-regular.}  A partition $\lambda=(\lambda_1,\cdots,\lambda_n)\in\Lambda^+(n,r)$ is called
chamber $e$-regular provided that, for all $i,j$ satisfying $1\leq i<j\leq n$, it holds that $\lambda_i -i\not\equiv \lambda_j-j$
mod$\,e$. Clearly, the set
$\Lambda^{+,\sharp}(n,r)$ of all chamber $e$-regular partitions is a (usually proper) subset of $\Lambda^+_{{\text{\rm $e$-reg}}}(n,r)$.

Let $\wH_r$ be the Hecke algebra over ${\mathbb Z}[\q,\q^{-1}]$ for the
Coxeter group $W={\mathfrak S}_r$ (symmetric group of degree $r$) with fundamental reflections $S=\{s_1,\cdots, s_{r-1}\}$, $s_i=(i,i+1)$. Here $\q$ is an indeterminate.
Each composition $\lambda\in\Lambda(n,r)$ defines a (right) $\q$-permutation module $x_\lambda \wH_r$, and the $\q$-Schur algebra over ${\mathbb Z}[\q,\q^{-1}]$
is the endomorphism algebra
\begin{equation}\label{schur}
\widetilde S_\q(n,r):=\End_{\wH_r}\left(\bigoplus_{\lambda\in\Lambda(n,r)}x_\lambda\wH_r\right).\end{equation}
(See \cite{DPS1} and \cite{DDPW} for more details.) The space $\bigoplus_{\lambda\in\Lambda(n,r)}x_\lambda\wH_r$ identifies with $V^{\otimes r}$ ($V$ a free ${\mathbb Z}[\q,\q^{-1}]$-module of rank $n$) which
has a natural right $\wH_r$-action.
Definition (\ref{schur}), due to Dipper and James \cite{DJ}, behaves well with respect to base change to any commutative ring $\sR$ over ${\mathbb Z}[\q,\q^{-1}]$. In particular, we can take $\sR=K={\mathbb Q}(\zeta)$, specializing $\q\mapsto \zeta^2$ and put $S_q(n,r):=\widetilde S_\q(n,r)\otimes_{{\mathbb Z}[\q,\q^{-1}]}K$. There is a surjective homomorphism $U_\zeta({\mathfrak{gl}}_n)\twoheadrightarrow S_q(n,r)$ \cite[Thm. 14.24]{DDPW} (see also \cite[Thm. 6.3]{DPS1} and \cite{Du}). Here $U_\zeta({\mathfrak{gl}}_n)$ is the quantum enveloping algebra over ${\mathbb Q}(\zeta)$
corresponding to the general linear Lie algebra ${\mathfrak{gl}}_n$; the action of $U_\zeta({\mathfrak{gl}}_n)$ on ``quantum tensor space" $V^{\otimes r}$
comes about by the action of $U_\zeta({\mathfrak{gl}}_n)$ on its natural module $V$ of dimension $n$. (We will argue momentarily  that there is a similar surjection
using $U_\zeta({\mathfrak{sl}}_n)$ for $\mathfrak{sl}_n$.)

The $q$-Schur algebra $S_q(n,r)$ is a QHA with poset $\Lambda^+(n,r)$. The arguments in \S8 can be applied (with some care)
to obtain similar results about $S_q(n,r)$, viewing the latter as a quotient of $U_\zeta({\mathfrak{gl}}_n)$.
However, it is also possible to use ${\mathfrak{ sl}}_n$ and quote the results of \S8 more directly. We briefly
describe this ${\mathfrak{ sl}}_n$ approach. 

 Label the simple roots $\Pi=\{\alpha_1,\cdots, \alpha_{n-1}\}$ of $\Phi$ (a root system of type $A_{n-1}$) in the usual way.  Fix $r$ and let $\Gamma_{n,r}$ be the subset (poset ideal) of $X^+$ consisting of dominant weights $\lambda=\sum_{i=1}^{n-1} a_i\varpi_i$ such that $r\geq\sum ia_i\equiv r$
mod$\,n$.
 Given $\lambda\in\Lambda^+(n,r)$, put $\overline{\lambda}=\sum_{i=1}^{n-1} (\lambda_i-\lambda_{i-1})\varpi_i\in X^+$.
Then $\lambda\mapsto\overline\lambda$ is a poset isomorphism $\Lambda^+(n,r)\to \Gamma_{n,r}$. Under
this isomorphism, the set $\Lambda^{+,\sharp}(n,r)$ of chamber $e$-regular weights corresponds to the set of $e$-regular elements in
$\Gamma_{n,r}$.

Now form the QHA algebra $A:=U_{\Gamma_{n,r}}$ (as above Lemma \ref{DSlemma}, taking $U_\zeta=U_\zeta({\mathfrak{sl}}_n)$ and $\Gamma=\Gamma_{n,r}$, allowing posets of non-regular weights).  For convenience, also denote $A$ by $A_{\Gamma_{n,r}}$. Then the dimension of $A_{\Gamma_{n,r}}$ is the sum of squares of dimensions of standard modules, and it is found to be the same as the dimension of  $S_q(n,r)$.   
The inclusion $U_\zeta(\mathfrak{sl}_n)\subseteq U_\zeta(\mathfrak{gl}_n)$ induces an algebra
homomorphism $A_{\Gamma_{n,r}}
\to S_q(n,r)$, and an exact functor $T:S_q(n,r){\text{\rm --mod}}\to A_{\Gamma_{n,r}}{\text{\rm --mod}}$ which
is described explicitly as the restriction of $U_\zeta(\mathfrak{gl}_n)$-modules (with composition factors having
highest weights in $\Lambda^+(n,r)$) to $U_\zeta(\mathfrak{ sl}_n)$-modules (with composition factors having
highest weights in $\Gamma_{n,r}$). The functor $T$ takes
standard, costandard, and irreducible modules to standard, costandard, and irreducible modules, respectively,
and it is compatible with the poset isomorphism. Therefore, the Comparison Theorem \cite{PS-2}, \cite[Thm. 2]{CPS8} implies that
$T$ is a category equivalence. Since the functor $T$ here arises as a ``restriction" through a map of algebras, the latter algebra map
must be an isomorphism.

Let $u'_\zeta$ be the sum of the regular blocks in the small quantum group $u_\zeta({\mathfrak{sl}}_\zeta)$.

Now Theorems \ref{KoszulforgeneralizedqSchur}, \ref{gradingresult}, and Corollary \ref{inverse} imply the following result.

\begin{thm}\label{qSchurtheorem} Let $n,r$ be  positive integers, $e> n$ an odd integer.
If $B$ is any block of $S_q(n,r)$ corresponding to a subset (poset ideal) of
 $\Lambda^{+,\sharp}(n,r)$ of chamber $e$-regular partitions, then $\gr B$ is a QHA with a Koszul
grading. Also, $B^!\cong (\gr B)^!$ is a Koszul algebra (and a QHA). In addition, for 
 $\lambda\in\Lambda^{+,\sharp}(n,r)$, the standard (Weyl) module $\Delta_q(\lambda)$ has a positive grading, generated in grade 0,
for the Koszul algebra $u'_\zeta$. 

The standard modules for the algebra $\gr B$ are the modules $\gr \Delta_q(\lambda)$, $\lambda\in
\Lambda^{+,\sharp}(n,r)$, and the multiplicities of the irreducible modules in the various sections
$\rad^i\Delta_q(\lambda)/\rad^{i+1}\Delta_q(\lambda)$ are coefficients of inverse Kazhdan-Lusztig
polynomials. 
\end{thm}

Again, the only need for the assumption that $e$ is odd is to be able to quote \cite{AJS} as in footnote 8.

As in Theorem \ref{KoszulforgeneralizedqSchur},  $\Bmod$, $\gr B$-mod, and
$\gr B$-grmod have compatible dualities. The modules $\gr\Delta_q(\lambda)$ are linear, and $\gr B$-mod has a graded Kazhdan-Lusztig theory (cf. \S4).  Both $\Bmod$ and $\gr B$-mod have a
satisfy the property (KL) with the same (classical) Kazhdan-Lusztig polynomials.  

\begin{rem} It is interesting to compare these results on $q$-Schur algebras with Ariki's recent work \cite{Ariki}.
Ariki obtains a $\mathbb Z$-grading on the $q$-Schur algebra $S_q(n,r)$, assuming that $q=\zeta^2$ is a primitive $e'$th root
of unity with $e'\geq 4$. (Thus, $e'=e$ if $e$ is odd, and $e'=e/2$ if $e$ is even (allowed in this case).) In addition, when $n\geq r$, he obtains a formula
for $[\Delta_q(\lambda):L_q(\nu)(r)]$ similar to that given in Corollary \ref{inverse}(b) for the multiplicity of $L_\zeta(\nu)$
in $F^i(\lambda)/F^{i+1}(\lambda)$. Each standard module $\Delta_q(\lambda)$ is given a specific $\mathbb Z$-grading. Since Corollary \ref{inverse} requires that $e>h=n$, and since the Hecke algebra $H_r$ (obtained from $\wH_r$ by specializing $\q$ to
$q$), and hence $S_q(n,r)$,
are semisimple if $e'>r$, Ariki's results and those above apply to different situations.  

In the preprint \cite{PS6},   we  eliminate the
hypothesis $e >n$ from Corollary \ref{inverse}(b) in many cases, and completely
so in type ${\mathfrak{sl}}_n$. Thus, the comparison of the two theories becomes more
meaningful.
In particular, one can ask if there is an adjustment (regrading)
of Ariki's graded
$q$-Schur algebra $A$ to give a graded isomorphism $A \cong \gr A$, or even ask if there is
a regrading making $A$ Koszul.  \end{rem}

\section{Representations in positive characteristic of simple algebraic groups }
Let $p$ be a prime integer, and continue to assume that $\Phi$ is a root system in the Euclidean
space $\mathbb E$ as in \S8. Use the notations of \S8 involving  $\Phi$,
taking $e=p$ in this section. For example, $W_p:=W\ltimes p{\mathbb Z}\Phi$, etc.  
The Jantzen region in $X^+$ is   the set
$$\Jan:=\{x\in X^+\,|\,(x+\rho,\alpha_0^\vee)\leq p(p-h+2)\}.$$
As in \S8 (but taking $e=p$), $\Res$ denotes the poset ideal in $X^+$ generated by the set $X^+_1:=X^+_{1,p}$ of
$p$-restricted dominant weights. Put $C^-:=\{x\in{\mathbb E}\,|\, (x+\rho,\alpha^\vee)\leq 0,\alpha\in\Pi, (x+\rho,\alpha^\vee_0)>-p\}$, the fundamental anti-dominant alcove. If $p\geq 2h-3$, $X^+_1\subset\Jan$, so that $\Res$ is a poset ideal
 in $\Jan$. We will always assume (unless otherwise explicitly noted) that $p\geq h$; thus, $-2\rho\in C^-$; and, in particular, $\Reg\not=\emptyset$.
Given a $p$-regular dominant weight $\lambda$, write $\lambda=w\cdot\lambda^-$ for $w=w_\lambda\in W_p$, $\lambda^-\in C^-\cap X$.  (The element $w_\lambda$ is uniquely determined.) We define $l(w)=l(w_\lambda)$ to be the length of $\lambda$. We will use this length function on $\Reg$ or its ideals below
(and as the length function for a (KL) property).

 Let $G$ be a simple simply connected algebraic group over an algebraically closed field $K$ of positive
 characteristic $p$. Assume $G$ defined and split over the prime
field ${\mathbb F}_p$.  Then $G$ has root system denoted $\Phi$ with
respect to a maximal split torus $T$. For $\lambda\in X^+$, let $L(\lambda)$ be the irreducible 
rational $G$-module of highest weight $\lambda$. Let $U={\text{\rm Dist}}(G)$ be the distribution (hyperalgebra) of $G$. The category of rational $G$-modules is equivalent to
the category of locally finite $U$-modules.

 For a finite ideal $\Gamma$ in the poset of $p$-regular
dominant weights, define (just as in (\ref{Asubgamma}) for $U=U_\zeta$) the algebra $U_\Gamma$
to be the quotient $U_\Gamma=U/J_\Gamma$, where $J_\Gamma$ is the annihilator of 
$U{\text{\rm --mod}}[\Gamma]$ (the full subcategory of finite dimensional rational $G$-modules generated by the
$L(\gamma)$, $\gamma\in\Gamma$). As in the quantum case of \S8, $A_\Gamma$ is a QHA with
weight poset $\Gamma$. (For instance, one needs only to mimic the argument in \cite[Prop. 3.5]{DS} mentioned
in \S8.)

Let $F:G\to G$ be the Frobenius morphism with kernel $G_1$, the
first infinitesimal subgroup of $G$. If $V\in\Gmod$ (the category of rational $G$-modules), let $V^{(1)}:=F^*V$, the
pullback of $V$ through $F$. If $\nu_0\in X_1^+$, let $Q(\nu_0)$ be the projective indecomposable cover of the irreducible $G_1$-module $L(\nu_0)|_{G_1}$.  If $p\geq 2h-2$, then $Q(\nu_0)$ has a unique compatible structure as a rational $G$-module \cite[\S11.1]{JanB}.
In fact, this module identifies with the indecomposable tilting module $T(2(p-1)\rho+w_0\nu_0)$ of highest weight $2(p-1)\rho+w_0\nu_0$ in $\Gmod$.

Let $\Xi\subseteq X^+$. Define $\Xi_1\subseteq X^+$ as follows: given $\lambda\in \Xi$, write
$\lambda=\lambda_0+p\lambda_1$, where, as usual, $\lambda_0\in X^+_1$, $\lambda_1\in X^+$.
Now put $\Xi_1=\{\lambda_1\,|\,\lambda\in\Xi\}$. 
Put 
$$a_1(\Xi):=\max_{\gamma\in \Xi_1}\{(\gamma,\alpha_0^\vee)\}.$$

The reader may check that, if $\lambda\in X^+$ belongs to $W_p\cdot 0$ and if $(\lambda,\alpha_0^\vee)
<2p-2h+2$, then $\lambda=0$. In this case, $H^1(G,L(\lambda))=0$. For the definition of $\Lambda(0)$
in the lemma below, see the discussion below (8.3.1) for $\Psi(0)$.

\begin{lem}\label{newnewlemma} Assume that $p\geq 2h-2$. Let $\Lambda$ be a finite non-empty poset ideal
in $X^+$ or $\Reg$. 

(a) If $a_1(\Lambda)< p - h+1$, then 
$$\Ext^1_G(Q(\mu_0)\otimes L(\mu_1)^{(1)},L(\lambda))=0,\quad\forall\lambda,\mu\in \Lambda$$

(b)  If 
$a_1(\Lambda) + a_1(\Lambda(0))<2p-2h+2$, then, for each $\mu\in\Lambda$,
 the module $Q(\mu_0)\otimes L(\mu_1)^{(1)}$ is the
projective cover of $L(\mu)$ in $U_{\Lambda(0)}$-mod.  
\end{lem}

\begin{proof} First, we prove (b).
The construction of $\Lambda(0)$ guarantees that the highest weight
of $Q:=Q(\mu_0)\otimes L(\mu_1)^{(1)}$ belongs to $\Lambda(0)$. Thus, $Q $ is
a $U_{\Lambda(0)}$-module whose head is obviously
  $L(\mu)$. To prove that $Q$ is projective as a $U_{\Lambda(0)}$-module, we need only check that
$\Ext^1_G(Q,L(\nu))=0$ for any $\nu\in\Lambda(0)$. This is equivalent to showing that 
$\Ext^1_G(Q(\mu_0),L(p\mu_1)^*\otimes L(\nu))=0,$ $\forall\nu\in\Lambda(0).$
The hypothesis guarantees that all the composition factors of $L(p\mu_1)^*\otimes L(\nu)$ 
have the form $L(\tau)^{(1)}\otimes L(\nu_0)$ with the only possible $\tau\in W_p\cdot 0$ being $0$
itself.   The
required vanishing follows from a standard Hochschild-Serre spectral sequence argument, using the normal subgroup scheme $G_1$ of $G$. Thus, (b) holds. 
Finally, (a) follows from a similar vanishing argument, replacing the condition ``$\nu\in\Lambda(0)$" with ``$\nu\in\Lambda$".\end{proof}
  
  Elementary methods for checking the hypotheses of Lemma \ref{newnewlemma} will be
  given in Lemma \ref{blowingup} and Corollary \ref{firstcorollary}, as well as (the proof of) the corollaries below.

\begin{cor}\label{cortonewlemma}  Assume that $p\geq 2h-2$. Let $\Lambda$ be a
 finite non-empty poset ideal in $X^+$ or $\Reg$ satisfying $a_1(\Lambda)<p-h+1$ (as
 in (a) of the above lemma). 
 If $A=A_{\Lambda}$ and if $\fa$ denotes the image of  $u$ in $A$, then  $(\rad\fa)A=\rad A$. In particular, the latter equality holds
 if $\Lambda$ is any finite non-empty poset ideal in $\Res$ or $\Resreg$. \end{cor}
 
\begin{proof} First,  $a_1(\Lambda)\leq\frac{p-1}{p}(h-1)<h-1$ if $\Lambda$ is contained in $\Gamma_{\text{\rm res}}$.  If $p\geq 2h-2$, then $a_1(\Lambda)< p - h+1$.  Thus, the second assertion of the
corollary follows from the first. To prove the first assertion, Lemma \ref{newnewlemma}(a)  shows that, for each $\mu\in\Lambda$, the largest quotient $Q_\Lambda$ of $Q=Q(\mu_0)\otimes L(\mu_1)^{(1)}$ with composition factors having highest weights in
$\Lambda$ is the projective cover of $L(\mu)$ in $A$-mod. However, we clearly have
$(\rad u)Q=\rad Q$, the latter taken in the category $G$-mod. Thus, $\rad Q_\Lambda=(\rad \fa)Q_\Lambda$
in $A$-mod. The left $A$-module $A$ is a direct sum of modules $Q_\Lambda$, with
$\mu$ varying over elements of $\Lambda$ (allowing repetitions). Thus, $\rad A=(\rad \fa)A$, completing
the proof. (We remark that much of the argument parallels that for the quantum case (the last assertion
of Lemma 8.3).)
\end{proof}

If $\Lambda\subseteq\Reg$, then $\fa$ is the also the image of $u'$, the sum of
the $p$-regular blocks in $u$.

For a given finite non-empty poset ideal $\Gamma$ of $X^+$ or $\Reg$, let $N=N(\Gamma)$ be the global dimension of  $A=U_\Gamma$.
Then $A$ is a QHA, so $N$ is a non-negative integer. 
An estimate on $N$, assuming $\Gamma\subseteq\Jan$ consists of $p$-regular weights and the LCF holds on $\Gamma$ can be obtained as follows.   First, 
 $$N=\max\{n\,|\, \Ext^n_G(L(\mu),L(\nu))\not=0,\,\,\,\mu,\nu\in\Gamma\}.$$
 By \cite[Thm. 3.5]{CPS1},  $N$ is the maximum $n_1+n_2$ for which $\Ext^{n_1}_G(L(\mu),\nabla(\tau))\not=0$ and $\Ext_G^{n_2}(\Delta(\tau),L(\nu))\not=0$ for some $\tau\in\Gamma$. 
The dimensions of $\Ext^m_G(\Delta(\tau),L(\nu))$ and  $\Ext_G^m(L(\nu),\nabla(\tau))$ are zero if $\nu\not\in
W_p\cdot\tau$. Otherwise, this (common) dimension is the coefficient of $q^{\frac{l(w)-l(y)-m}{2}}$ in a Kazhdan-Lusztig polynomial
$P_{y,w}(q)$ for $y,w\in W_p$ with $y^{-1}\cdot\tau= 
w^{-1}\cdot\nu\in C^-$.  Obviously, the maximum value of $m$ can be no more than the maximum
value of $l(w)-l(y)$, with $y^{-1}\cdot\tau, w^{-1}\cdot\nu\in C^-$. The element $y$ of minimal length such that
$y^{-1}\cdot\tau\in C^-$, for some $\tau\in \Gamma$, is $y=w_0$. If $\nu\in\Gamma$ with $w^{-1}\cdot\nu
\in C^-$, we have $w=w_0d$ with $l(w)=l(w_0)+l(d)$, computing lengths with respect to $C^-$ (i.~e.,
taking the fundamental reflections for the Coxeter group $W_p$ to be those in the walls of $C^-$). The
number $l(d)$ can be computed explicitly in terms of $\nu$: In the spirit of \cite[II, 6.6]{JanB}, for any
$\lambda\in\Reg$, let ${\bold d}(\lambda):=\sum_{\alpha\in\Phi^+}n_\alpha$ with
$pn_\alpha<(\lambda+\rho,\alpha^\vee)<p(n_\alpha+1)$. For $\nu\in\Gamma$, $w^{-1}\cdot\nu\in C^-$,
we can now calculate the length of $d$ above. Define $w':=w_0dw_0$, so that $(w')^{-1}\cdot\nu
\in w_0\cdot C^-=: C^+$. Of course, $0\in C^+$. Then ${\bold d}(\nu)$ counts the number of hyperplanes of the
form $\{x\,|\,(x+\rho,\alpha^\vee)=pm\}$, $m\in\mathbb Z$, which separate $\nu$ from 0, equivalently, 
$\nu$ from $(w')^{-1}\cdot \nu\in C^+$. This number is the length $l'(w')$, computed taking the
fundamental reflections to be in the walls of $C^+$. Thus, $l'(w')=l(w_0w'w_0)=l(d)$. So  
that $l(d)={\bold d}(\nu)$.

\begin{prop}\label{charp} Let $\Gamma$ be a finite non-empty poset ideal of $p$-regular weights contained in $\Jan$ for which the LCF holds. Let  $N=N(\Gamma)$ be the global dimension of $A=U_\Gamma$.  Then 
$$N\leq 2\max_{\nu\in\Gamma}\{{\bold d}(\nu)\}.$$
(In case $\Gamma=\Resreg:=\Res\cap \Reg$, $N\leq 2{\bold d}((p-2)\rho)\leq (h-1)|\Phi|$.)
 \end{prop}

 Hypothesis \ref{section7hypothesis} requires the following result. Recall that
if $\emptyset\not=\Psi\subseteq\Reg$ and if $m\geq -1$, the poset ideal $\Psi(m)$ is defined after (\ref{definitionoff}) in \S8.

\begin{lem}\label{blowingup} Let $\Lambda$ be a finite non-empty poset ideal in $X^+$ or $\Reg$.
Then for any integer $m\geq -1$,  
$$a_1(\Lambda(m))\leq a_1(\Lambda) + 2(m+1)(h-1)$$
with strict inequality whenever $m\geq 0$.
 \end{lem}
\begin{proof} Clearly, the inequality  holds for $m=-1$. By induction, it suffices to treat the case $m=0$.  By
definition, $\Lambda(0)$ is the poset ideal generated by the weights 
${\mathfrak f}_p(\lambda)=2(p-1)\rho + w_0\lambda_0 + p\lambda_1= (p-2)\rho + w_0\lambda_0+p(\lambda_1+\rho)$, $\lambda\in\Lambda$. Let $\gamma$ be a dominant weight such that $\gamma\leq {\mathfrak f}_p(\lambda)$. Thus,
$$p(\gamma_1,\alpha^\vee_0)\leq (p-2)(\rho,\alpha_0^\vee) + pa_1(\Lambda) +p(h-1).$$
Dividing by $p$ gives
$(\gamma,\alpha_0^\vee)< h-1 + a_1(\Lambda) + h-1=a_1(\Lambda)+ 2(h-1),$
as required.\end{proof}

The following corollary is  an easy consequence of Lemma \ref{blowingup}.  
 
\begin{cor}\label{firstcorollary} Let $\Lambda$ be a finite non-empty poset ideal in $\Res$ or in $\Resreg$. If $m\geq-1$, 
then
$$a_1(\Lambda(m))< (2m+3)(h-1).$$
In addition, if $p\geq (2m+3)(h-1)$, then 
  $$a_1(\Lambda(m-1))+a_1(\Lambda(m))< 2p-2h+2.$$
\end{cor}

The positive characteristic version of Theorem \ref{KoszulforgeneralizedqSchur}
requires that the LCF holds on $\Reg\cap\Jan$. By Andersen-Janzten-Soergel \cite{AJS}
this formula holds if $p\gg 0$ (depending on $\Phi$). Some specific bounds are provided in \cite{Fiebig}. If $\Gamma$
is a finite non-empty poset ideal in the poset $\Resreg$, there is a finite dimensional algebra $B=U_\Gamma$,
which is a quotient of 
the distribution (Hopf) algebra $U=\text{Dist}(G)$ of $G$ such that $B$-mod identifies with the full subcategory of
finite dimensional rational $G$-modules which have composition factors $L(\gamma)$, $\gamma\in\Gamma$. 
This fact follows from \cite[\S\S2,3]{DS} (which was quoted \S8 for the similar quantum result), but
the reader may wish to consult the earlier treatment in \cite[\S3.2]{Do1} by Donkin (where $\text{Dist}(G)$
is denoted $\text{hy}(G)$ and called the hyperalgebra of $G$). 
Necessarily,
the algebra $B$ is a QHA with standard, costandard, and irreducible modules the corresponding
 modules  in $G$-mod.   Of course, the restricted enveloping algebra $u$
of $G$ is a Hopf subalgebra of $\text{Dist}(G)$.

We remark that the categories $\Bmod$, $\gr B$-mod, and
$\gr B$-grmod all have compatible dualities for the algebras $B=U_\Gamma$ below for any poset. See 
the parallel discussion in \S8.

\begin{thm}\label{charpversion} Assume that $G$ is a simple simply connected algebraic group over a field of characteristic $p\geq 2h-2$. Also, assume that the LCF holds for all $p$-regular weights in $\Jan$.  Let $\Gamma$ be a poset ideal in $\Resreg$, and let $N$ be the global dimension of $B:=U_\Gamma$ or $N=2$ in case
this global dimension is 1. (By Proposition \ref{charp}, $N\leq (h-1)|\Phi|$.)  Assume also that $p\geq 2N(h-1)-1$.  
Then $\gr B$ has a graded Kazhdan-Lusztig theory. In particular, the algebra $\gr B$ is a QHA with a Koszul
grading. Its standard modules  have the form  $\gr\Delta(\lambda)$,
$\lambda\in\Gamma$, and they are linear in $\gr B$-grmod. \end{thm}

\begin{proof} Let $u'$ be the sum of the $p$-regular blocks of the restricted enveloping algebra $u$ of $G$. Since the LCF holds and $p>h$ (excluding $p=h=2$, a trivial case), the algebra $u'$ is a Koszul algebra by \cite[Prop. 18.17, p. 256]{AJS}. If $B$ has global dimension 0, the theorem holds trivially, so we may assume that $N\geq 2$. Define
$\Lambda=\Resreg(N-2)$.
 Set $A:=U_\Lambda$. If $ \lambda_0\in X^+_1\cap \Res$, the condition that $p\geq 2h-2$ means that  the $u'$-projective
cover $Q(\lambda_0)$ of $L(\lambda_0)$ is a rational $G$-module. In addition, $\Resreg(0)\subseteq
\Lambda$, so that $Q(\lambda_0)\in\Amod$.\footnote{The condition $\Resreg(0)\subseteq \Lambda$ is
similar to and plays the same role here as the assumption that $\Lambda$ be $e$-fat in \S8.} So $u'\to A$ is an injection, mapping
the Koszul algebra $u'$ isomorphically onto its image $\fa$ in $A$.  

The condition that $p\geq 2N(h-1)-1$ is readily checked, using Lemma \ref{blowingup}, to verify that $a_1(\Lambda)<p-h+1$, so Corollary
\ref{cortonewlemma} implies that $(\rad\fa)A=\rad A$. Thus,  Hypothesis \ref{hypothesisofsection6} holds for $(A,\fa)$. 
Also, Lemma \ref{newnewlemma} and Corollary \ref{firstcorollary} imply that condition (2) of Hypothesis \ref{section7hypothesis} holds. 
Note that $\Gamma\subseteq\Lambda$. By assumption, $\Bmod$ satisfies the (KL) property.  Theorem \ref{Koszultheorem} implies that $B$ satisfies
(SKL$'$). Now the theorem follows from the discussion of \S4 on graded Kazhdan-Lusztig theories.
 \end{proof}

\begin{rems}\label{finaleremarks} (a) We expect that the condition that $p\geq 2N(h-1)-1$ in 
Theorem \ref{charpversion} can be removed, just leaving the conditions that $p\geq 2h-2$ and that $p$ is large enough that the LCF holds for all $p$-regular weights in $\Jan$.
The improvement should be a 
consequence of our studies \cite{PS5.5} and \cite{PS5}  of integral versions of the algebras $\gr B$ and $p$-filtrations. The conclusion in Theorem \ref{charpversion} that $\gr B$ is a QHA can already be obtained in this
paper under weaker conditions; see Theorem \ref{alternativetheorem} below.

 Of course, current bounds \cite{Fiebig} on $p$ required for the validity of the LCF are much larger even than the current requirement that $p\geq 2N(h-1)-1$ in
Theorem \ref{charpversion}.

(b) Under the same hypotheses as those of Theorem \ref{charpversion}, the analogue of Corollary \ref{inverse} holds. In particular, the homological dual
$B^!:=\Ext^\bullet_G(B/\rad B, B/\rad B)$,
is a Koszul QHA algebra, isomorphic to the homological dual of $\gr B$, $B=U_\Gamma$. In addition, for
$\lambda\in\Gamma$, the
multiplicities $[\gr\Delta(\lambda):L(\mu)(n)]$ in the radical filtration of $\Delta(\lambda)$ are given explicitly
by inverse Kazhdan-Lusztig polynomial coefficients.\footnote{A similar multiplicity result for an explicit semisimple series 
  can be obtained
by reduction mod $p$ from the quantum radical filtration in \S8 for Weyl modules $\Delta(\lambda)$ with $\lambda\in\Jan$ a regular weight, when $p>h$ and the LCF holds on $\Jan$. Parity considerations imply the
needed complete reducibility.  (No claim is made here that this series is
the characteristic $p$ radical series.) As in Corollary \ref{inverse}(b) and its proof, this semisimple  series can be translated to an explicit semisimple series for Weyl modules associated to singular weights in $\Jan$.
} (A generic version of this radical filtration multiplicity result
had been proved in \cite{AK} and \cite{Lin3}.)  

(c) To obtain an analog for the Schur algebra $S(n,r)$, fix $n$ and choose
$p$ large enough so that the LCF holds for all regular weights in
the Jantzen region, for $G$ of type $SL_n$. Let $N=h|\Phi|=n^2(n-1)/2$, and also assume that
$p\geq 2N(n-1)-1$. (As in (a), we expect this condition can be replaced by $p\geq 2n-2$.)  Now let $r\geq p$ be any integer such that (in the notation of \S10),
the regular weights in $\Gamma_{n,r}$ are contained in $\Gamma_{\text{\rm res,reg}}$. (For example,
 $r-p<n$ works.)  Then the above result is applicable for any regular block (or product of blocks) of $B_{\Gamma_{n,r}}\cong
 S(n,r)$ (the latter isomorphism following as in \S10 for the $q$-Schur algebra).

Potentially, such a characteristic $p$-analogue of Theorem \ref{qSchurtheorem}  holds without any $p$-chamber regularity condition on the weights, or any condition on $p$, other than perhaps $p^2>r$---the hypothesis of the   James conjecture \cite{Jam}).
However,  such a generalization to singular weights is beyond our current methods. (Computer calculations by Carlson, found at
${\text{\rm {\tt{http://www.math.uga.edu/$\sim$jfc/schur.html}}}}$,
support the speculation that $\gr S(n,r)$ is Koszul under the hypothesis that $p^2>r$.)
\end{rems}

 In the following result, Theorem \ref{grBisQHA} is used to improve the
 bound on $p$ in Theorem \ref{charpversion} as regards the conclusion that $\gr B$ is a QHA.
 
\begin{thm}\label{alternativetheorem} Assume that $G$ is a simple simply connected algebraic group over a field of characteristic $p\geq 4h-5$. Also, assume that the LCF holds for all $p$-regular weights in $\Jan$.  Let $\Gamma$ be a poset ideal in $\Resreg$ and let  $B:=U_\Gamma$. Then $\gr B$ is a QHA with weight poset $\Gamma$ whose standard modules are the $\gr\Delta(\lambda)$,
$\lambda\in\Gamma$. \end{thm}

We postpone the proof of this theorem until the following result is established. 

 \begin{thm}\label{filteringWeyl} (a) Assume that $G$ is a simple simply connected algebraic group over a field of characteristic $p\geq 4h-5$. Assume also that
the LCF holds for all $\lambda\in\Gamma_{\text{\rm res,reg}}\subseteq\Jan$. For any
 $\lambda\in\Gamma_{\text{\rm res,reg}}$, there is a positive grading
$\Delta(\lambda)=\bigoplus_{n\geq 0}\Delta(\lambda)_n$
of $\Delta(\lambda)$ as a module for the Koszul algebra $u' $. As a graded $u'$-module, $\Delta(\lambda)$
is generated in grade 0. As a $u'/\rad u'$-module
$\Delta(\lambda)_0\cong L(\lambda)\cong L(\lambda_0)\otimes L(\lambda_1)^{(1)}.$

(b) Additionally, assume that the full (i.~e., not just the regular part $u'$) restricted enveloping algebra $u$ is Koszul (which holds for $p\gg h$ by \cite{Riche}).
Then for any $\lambda\in\Jan$, $\Delta(\lambda)$ has a positive grading as a module for $u$. As a graded $u$-module,
$\Delta(\lambda)$ is generated in grade 0. Also, as a $u/\rad u$-module, $\Delta(\lambda)_0\cong L(\lambda)\cong L(\lambda_0)\otimes L(\lambda)^{(1)}$ (the irreducible
$G$-module of highest weight $\lambda$).
\end{thm}

\begin{proof} (a): The condition $p\geq 4h-5\geq 2h-2$ guarantees  
  that the $u$-projective cover $Q(\lambda_0)$ of $L(\lambda_0)$ is a $G$-module for  $\lambda_0\in X^+_1$.
 (See the discussion above Lemma \ref{newnewlemma}.) Let $\Gamma=\Resreg$ and $\Lambda=\Gamma(0)$.  Also, Lemma \ref{newnewlemma}(b) shows that $Q=Q(\lambda_0)\otimes
 L(\lambda_1)^{(1)}$ is the projective cover of $L(\lambda)$ in $A_{\Lambda}$-mod for any $\lambda\in\Gamma$.  Observe that $a_1(\Gamma) + a_1(\Lambda)= a_1(\Gamma)+a_1(\Lambda(0)) < 2p-2h+2$ by Corollary \ref{firstcorollary} since $p\geq 4h-5
 \geq 3h-3$. So, by Lemma \ref{newnewlemma}(b), $Q:=Q(\lambda_0)\otimes L(\lambda_1)^{(1)}$ is the projective
 cover in $A:=A_\Lambda$-mod of $L(\lambda)$. 
  Of course, $Q$ is also $\fa$-projective, where $\fa$ is the isomorphic image of $u'$ in $A$.
 
 Finally, to apply Theorem \ref{MainInfinitesimalThm} with $A=A_{\Lambda}$ and $\lambda\in \Gamma$, it is necessary 
 to check that $(\rad a)A=\rad A$ (which is part (3) of Hypothesis \ref{hypothesisofsection6}). This follows from Corollary \ref{cortonewlemma}, since $a_1(\Lambda)< a_1(\Gamma)
 +2(h-1)\leq h-2+2(h-1)= (4h-5) - h+1\leq p-h+1$ by Lemma \ref{blowingup}. This proves (a).
 
 (b): The proof is similar.\end{proof}

\noindent
{\it  Proof of Theorem \ref{alternativetheorem}:} 
Let $\Lambda=\Resreg(0)$ and let $\Gamma$ be the given poset ideal. The proof of
Theorem \ref{filteringWeyl} shows Hypothesis \ref{hypothesisofsection6} holds for  $(A,\fa)$,
where $A=U_\Lambda$, provided $p\geq 4h-5$. The theorem follows from Theorem \ref{grBisQHA}. \qed

\medskip
In both the above two proofs, it was only needed that $\fa$ be tightly graded. We expect Theorem \ref{filteringWeyl}(a) can be
improved to include all $p$-regular dominant weights (at least when $p\geq 2h-2$ and the LCF holds in the Jantzen region).

\end{document}